\documentclass[reqno, letterpaper, oneside]{article}
\usepackage{amsmath,amsfonts,amssymb,amsthm}

\usepackage{bbm}
\usepackage{setspace}
\usepackage{geometry}
\usepackage{url}
\usepackage[square,numbers]{natbib}
\bibpunct[, ]{[}{]}{;}{n}{,}{,}

\title
{Sesqui-type branching processes}

\author{Svante Janson%
\thanks{Department of Mathematics, Uppsala University, 
PO Box 480, SE-751~06 Uppsala, Sweden. 
E-mail: {\tt svante.janson@math.uu.se}. 
Partly supported by the Knut and Alice Wallenberg Foundation.}
\ and Oliver Riordan%
\thanks{
Mathematical Institute, University of Oxford, Radcliffe Observatory Quarter, Woodstock Road, Oxford OX2\thinspace6GG, UK.
E-mail: {\tt riordan@maths.ox.ac.uk}.}
\ and Lutz Warnke%
\thanks{
School of Mathematics, Georgia Institute of Technology, Atlanta GA~30332, USA;  
{\em and\/}
Peterhouse, Cambridge CB2\thinspace1RD, UK.
E-mail: {\tt warnke@math.gatech.edu}.}}

\date{June 24, 2017}

\numberwithin{equation}{section}

\renewcommand\le{\leqslant}
\renewcommand\ge{\geqslant}

\allowdisplaybreaks

\theoremstyle{plain}
\newtheorem{theorem}{Theorem}[section]
\newtheorem{lemma}[theorem]{Lemma}

\theoremstyle{definition}

\newtheorem{definition}[theorem]{Definition}

\newtheorem{remark}[theorem]{Remark}

\theoremstyle{remark}

\newenvironment{romenumerate}[1][-10pt]{
\addtolength{\leftmargini}{#1}\begin{enumerate}
 \renewcommand{\labelenumi}{\textup{(\roman{enumi})}}%
 }{\end{enumerate}}

\newcounter{oldenumi}
{\setcounter{oldenumi}{\value{enumi}}
\begin{romenumerate} \setcounter{enumi}{\value{oldenumi}}}
{\end{romenumerate}}

\newcounter{thmenumerate}

\newcounter{xenumerate}   

\newcommand\pfitemx[1]{\par#1:}
\newcommand\pfitemref[1]{\pfitemx{\ref{#1}}}

\newcommand{\refT}[1]{Theorem~\ref{#1}}

\newcommand{\refL}[1]{Lemma~\ref{#1}}

\newcommand{\refS}[1]{Section~\ref{#1}}

\newcommand{\refD}[1]{Definition~\ref{#1}}

\begingroup
  \divide\count255 by 60
  \multiply\count255 by -60
  \advance\count255 by \time
  \ifnum \count255 < 10 \xdef\klockan{\the\count1.0\the\count255}
  \else\xdef\klockan{\the\count1.\the\count255}\fi
\endgroup

\newcommand{\sumno}{\sum_{n=0}^\infty}

\newcommand\set[1]{\ensuremath{\{#1\}}}

\newcommand\xpar[1]{(#1)}
\newcommand\bigpar[1]{\bigl(#1\bigr)}
\newcommand\Bigpar[1]{\Bigl(#1\Bigr)}
\newcommand\biggpar[1]{\biggl(#1\biggr)}
\newcommand\lrpar[1]{\left(#1\right)}

\newcommand\xcpar[1]{\{#1\}}

\newcommand\abs[1]{|#1|}
\newcommand\bigabs[1]{\bigl|#1\bigr|}
\newcommand\Bigabs[1]{\Bigl|#1\Bigr|}

\def\rompar(#1){\textup(#1\textup)}    

\newcommand\parfrac[2]{\lrpar{\frac{#1}{#2}}}

\def\xexp(#1){e^{#1}}

\newcommand\ntoo{\ensuremath{{n\to\infty}}}

\newcommand\norm[1]{\|#1\|}
\newcommand\bignorm[1]{\bigl\|#1\bigr\|}
\newcommand\Bignorm[1]{\Bigl\|#1\Bigr\|}
\newcommand\downto{\searrow}

\newcommand\punkt{.\spacefactor=1000}    
    
\newcommand\ie{i.e\punkt}
\newcommand\eg{e.g\punkt}

\newcommand\cf{cf\punkt}

\newcommand\ii{\mathrm{i}}

\newcommand\bbR{\mathbb R}
\newcommand\bbC{\mathbb C}
\newcommand\bbN{\mathbb N}

\newcommand\bbZ{\mathbb Z}

\newcounter{CC}
\newcommand{\CC}{\stepcounter{CC}\CCx} 
\newcommand{\CCx}{C_{\arabic{CC}}}     
\newcommand{\CCdef}[1]{\xdef#1{\CCx}}     
\newcommand{\CCname}[1]{\CC\CCdef{#1}}    
\newcounter{cc}
\newcommand{\cc}{\stepcounter{cc}\ccx} 
\newcommand{\ccx}{c_{\arabic{cc}}}     
\newcommand{\ccdef}[1]{\xdef#1{\ccx}}     
\newcommand{\ccname}[1]{\cc\ccdef{#1}}    

\renewcommand\Re{\operatorname{Re}}

\newcommand\E{\operatorname{\mathbb E{}}}
\renewcommand\P{\operatorname{\mathbb P{}}}

\newcommand\Det{\operatorname{Det}}

\newcommand\sign{\operatorname{sign}} 

\newcommand\tr{\mathrm{tr}}

\newcommand\ga{\alpha}
\newcommand\gb{\beta}
\newcommand\gd{\delta}

\newcommand\gf{\varphi}
\newcommand\gam{\gamma}

\newcommand\gl{\lambda}

\newcommand\eps{\varepsilon}

\renewcommand\phi{\xxx}  

\newcommand\cD{\mathcal D}

\newcommand\cK{\mathcal K}

\newcommand\cN{\mathcal N}

\newcommand\cR{{\mathcal R}}

\newcommand\qw{^{-1}}
\newcommand\qww{^{-2}}

\newcommand\qqw{^{-1/2}}

\newcommand\dd{\,\mathrm{d}}
\newcommand\ddx{\mathrm{d}}

\newcommand\pddd[1]{\frac{\partial}{\partial#1}}

\newcommand{\pgf}{probability generating function}
\newcommand{\mgf}{moment generating function}
\newcommand{\chf}{characteristic function}

\newcommand\rhs{right-hand side}

\newcommand\bp{\mathfrak X}

\newcommand\bpil{\bp^{1,L}}
\newcommand\bpis{\bp^{1,S}}
\newcommand\xs{^S}
\newcommand\xl{^L}
\newcommand\tg{\tilde g}
\newcommand\tf{\tilde f}
\newcommand\intpipi{\int_{-\pi}^\pi}
\newcommand\kk{\cK}
\newcommand\kko{\kk^0}
\newcommand\kki{\kk^1}
\newcommand\bm{m}
\newcommand\an{n}
\newcommand\xm{^{(m)}}
\newcommand\gabuv{\ga+\ii u,\gb+\ii v}
\newcommand\nx{n^{-0.4}}
\newcommand\Nd{d}
\newcommand\xx{x^*}
\newcommand\yy{y^*}
\newcommand\xgagb{{(\ga,\gb)}}
\newcommand\xxuv{[u,v]}
\newcommand\Psix{\widehat\Psi}
\newcommand\qN{\P(|\bp|=N)}
\newcommand\ddy{\frac{\partial}{\partial y}}

\newcommand\pddu{\frac{\partial}{\partial u}}

\newcommand\ddt{\frac{\ddx}{\ddx t}}
\newcommand\ddu{\frac{\ddx}{\ddx u}}
\newcommand\SX{T}

\newcommand\bY{\bar Y}
\newcommand\bZ{\bar Z}
\newcommand\tc{t_{\mathrm{c}}}
\newcommand\bb[1]{\bigl(#1\bigr)}
\newcommand\bg{\bar g}
\newcommand\rhox{\sigma} 
\newcommand\YYZZ{{Y,Z,Y^0,Z^0}}
\newcommand\YY{{Y,Y^0}}
\newcommand\YZ{{Y,Z}}
\newcommand\trho{\tilde\rho}
\newcommand\hrho{\hat\rho}

\newcommand\rhoY{\rho_Y}
\newcommand\rhoYt{\rho_{Y_t}}
\newcommand\rhoYY{\rho_{\YY}}
\newcommand\rhoYYt{\rho\YYt}
\newcommand\Yu{_{Y_u}}
\newcommand\Yt{_{Y_t}}
\newcommand\YYu{_{Y_u,Y^0_u}}
\newcommand\YYt{_{Y_t,Y^0_t}}
\newcommand\hG{\widehat G}

\newcommand\CM{M}

\newcommand{\indic}[1]{\mathbbm{1}_{\{{#1}\}}}

\newcommand\uvrow{(u\ v)}
\newcommand\uvcol{\begin{pmatrix}u\\v\end{pmatrix}}
\newcommand\Bd{B^d}
\newcommand\Bt{B^2}

\let\OLDthebibliography\thebibliography
\renewcommand\thebibliography[1]{
  \OLDthebibliography{#1}
  \setlength{\parskip}{0pt}
  \setlength{\itemsep}{0pt plus 0.3ex}
}

\begin{document}

\renewcommand{\thefootnote}{\fnsymbol{footnote}}
\footnotetext{\hspace{-0.5em}AMS 2010 Mathematics Subject Classification: 60J80}
\renewcommand{\thefootnote}{\arabic{footnote}} 

\maketitle

\begin{abstract} 
We consider branching processes consisting of particles (individuals) of two types (type~$L$ and type~$S$)
in which only particles of type~$L$ have offspring, proving estimates for the survival probability
and the (tail of) the distribution of the total number of particles. Such processes are in some
sense closer to single- than to multi-type branching processes. Nonetheless, the second, barren,
type complicates the analysis significantly. The results proved here 
(about point and survival probabilities) are a key ingredient
in the analysis of bounded-size Achlioptas processes in a recent paper by the last two authors. 
\end{abstract}

\section{Introduction}

Throughout the paper we consider branching processes in which every particle is of one of two \emph{types},
called (for compatibility with the notation in~\cite{RW}), `type $L$' and `type $S$'. Particles of type $S$ may
be thought of as barren: they have no children. Each particle of type $L$ will have some random number 
of children of each type; as usual, we have independence between the children of different particles,
but the numbers $Y$ and $Z$ of type-$L$ and type-$S$ children of one particle need \emph{not} be independent.
The formal definition is as follows.

\begin{definition}\label{def:bp}
Let $(Y,Z)$ and $(Y^0,Z^0)$ be probability distributions on $\bbN^2$. We write
$\bp^1=\bp^1_{Y,Z}$  
for the Galton--Watson branching process started with a single particle of type~$L$, in which each particle of type~$L$ has $Y$ children of type~$L$ and $Z$
of type~$S$. Particles of type~$S$ have no children, and the children of different particles are independent. 
We write
$\bp=\bp_{Y,Z,Y^0,Z^0}$ for the branching process defined as follows: 
start in generation one with $Y^0$ particles of type~$L$ and $Z^0$ of type~$S$.
Those of type~$L$ have children according to~$\bp^1_{Y,Z}$, independently of each other
and of the first generation. Those of type~$S$ have no children.
We write~$|\bp|$ ($|\bp^1|$) for the total number particles in~$\bp$ $(\bp^1)$. 
\end{definition}

These branching processes are in some sense
essentially single-type: one could first generate the tree of type-$L$ particles as a classical single-type
Galton--Watson process, and then consider particles of type $S$. However, since the numbers of type-$S$ and type-$L$
children are not necessarily independent, this two-stage description does not seem particularly easy to work with.

The motivation for considering such processes (and in particular for allowing a different rule for the first generation)
comes from the application to studying the phase transition in Achlioptas processes in~\cite{RW}. 
Achlioptas processes are
evolving random graph models that have received considerable attention 
(see, e.g., \cite{DRS,RWep,BK,SW,JS,RWapcont,KPS,BBW12b,RWapsubcr} and the references therein). 
We shall say nothing further about 
 these random graph processes here,
aiming to keep the paper self-contained, and purely
about branching processes.

We shall prove two main results. 
Firstly, in Section~\ref{Smain}, we consider an individual branching process
of the type above, giving an asymptotic formula
for the point probability $p_N=\P(|\bp|=N)$ under certain conditions on the distributions $(Y,Z)$ and $(Y^0,Z^0)$.
This formula is proved in Sections~\ref{Sint}--\ref{Ssum}, which are the heart of the paper.
Then, in Section~\ref{Svary}, we consider families of processes where the offspring distribution
varies analytically in an additional parameter $t$. Roughly speaking, we show that the key quantities
in the formula in Section~\ref{Smain} then vary analytically in $t$. This result (which in particular
implies properties of the near-critical case) is needed in~\cite{RW}.
Finally, in Section~\ref{Ssurvival}, we prove corresponding results for the survival probability~$\P(|\bp|=\infty)$.
Here the barren type plays no role, so the results effectively concern single-type processes
and are much simpler.  

\begin{remark}
Although the definition of sesqui-type branching processes is adapted to the
application in~\cite{RW}, the results here are applicable, at least in principle, to a more
general class of branching processes. Consider a finite-type Galton--Watson process in
which there is one special type (type $L$), and all other types are `doomed' (lead to
finite trees of descendants a.s.). Such a process may be transformed into a sesqui-type
process in a natural way: for each type-$L$ particle replace its children of all doomed
types, and their (necessarily doomed) descendents, by type-$S$ children (keeping the same
total number of particles). For our results
to apply to the transformed process we need further conditions, roughly speaking that the
`doomed' subtrees are not too close to critical; but in outline, all processes with (at most)
one type that can potentially survive are covered. 
Branching processes of this type (with one doomed type) have been studied 
by several authors, giving various results different from ours; see for 
example~\cite{Sevastyanov, Vjugin, DrmotaVatutin}. 
\end{remark}

\subsection{Some notation and conventions}\label{Snot}

Throughout we write $\bbN:=\set{0,1,2,\dots}$ for the non-negative integers.

Given a two-dimensional random variable $(Y,Z)$ taking values in $\bbN^2$, we
denote its bivariate \pgf{} by
\begin{equation}
  g_{Y,Z}(y,z):=\E (y^Yz^Z) = \sum_{k,l\ge0}\P(Y=k,Z=l)y^kz^l,
\end{equation}
for all complex $y$ and $z$ such that the expectation (or sum) converges
absolutely. We will also consider the bivariate \mgf{}
\begin{equation}
f_{Y,Z}(y,z):=g_{Y,Z}\bigpar{e^y,e^z} = \E\bigpar{e^{yY+zZ}} .
\end{equation}
When considering a particular branching process as in Definition~\ref{def:bp}, 
we often write $g=g_{Y,Z}$ and $f=f_{Y,Z}$ for brevity.

We denote the coefficient of $y^kz^l$ in a power series $G(y,z)$ by
$[y^kz^l]G(y,z)$. 

We say that a function $f$ defined on $I\subseteq \bbC$ is analytic if for every $x_0\in I$ 
there is an $r>0$ and a power series $g(x)=\sum_{j\ge 0} a_j (x-x_0)^j$ with radius of convergence at least $r$ such
that~$f$ and~$g$ coincide on $(x_0-r,x_0+r)\cap I$. 
A function~$f$ defined on some domain including~$I$ is analytic on~$I$ if $f|_I$ is analytic. 
The definitions for functions of several 
real or complex variables are analogous.

If $f$ is an analytic function of $\Nd$ variables, defined in an open set 
$U\subseteq \bbC^\Nd$, we denote its derivative by~$Df$, and
its $m$th derivative by~$D^mf$. Note that~$D^m f$ is an analytic function
from $U$ to the linear space of all (symmetric)
$m$-linear forms $\bbC^\Nd\to\bbC$.
In particular, for each $z\in U$, 
$Df(z)$ is a linear form, which can also be regarded
as a vector (the usual gradient);
we write $D_if:=\frac{\partial}{\partial x_i}f$, so
$Df(z)=\bigpar{D_1f(z),\dots,D_\Nd f(z)}$.
Similarly,
$D^2f(z)$ is a bilinear form, which may be
regarded as a $\Nd\times \Nd$ matrix
with entries $D_{ij}f(z)$, where~$D_{ij}=D_iD_j$.
We denote its determinant by $\Det(D^2f(z))$.
(This is known as the Hessian of $f$.)

For a vector $x\in\bbC^\Nd$, let $D^mf(z)[x]$ denote $D^mf(z)(x,\dots,x)$, where the
vector $x$ is repeated $m$ times. 
When using coordinates $x=(u,v)$ in the case $d=2$, we  write $[u,v]$ for
$[(u,v)]$,
so, regarding $D^2f$ as a matrix and $x$ as a (column) vector, we have
\begin{equation}\label{Dmf}
 D^2f(z)[u,v] =  \uvrow \: D^2f(z) \: \uvcol .
\end{equation}

We denote the usual Euclidean norm of vectors by $|\cdot|$. For operators
and the multilinear forms $D^mf$ we use $\norm\cdot$ for the usual norm (any
other norm would do as well).

For real symmetric matrices, $A\le B$ means that $B-A$ is 
positive definite, i.e., that $v^\tr(B-A)v\ge 0$ for all real vectors $v$.
In particular, if $A$ is a $d\times d$ symmetric matrix and $c\in\bbR$, then
\begin{equation}\label{rM}
  A\ge cI \iff
v^\tr Av\ge c|v|^2 \text{ for all $v\in\bbR^\Nd$}.
\end{equation}

\begin{remark}
We adopt the following notational convention regarding constants. $c$ and $C$ are used `locally' 
(within a single proof), while numbered constants $c_1$, $C_1$ etc retain their meaning throughout 
the paper. The constants $c_i$, which are numbered in the order they are introduced, obey the inequalities 
\[
 c_1   \le c_8   \le c_7 \le c_6 \le c_5  \quad\text{and}\quad c_4\le c_2
\]
We write $y$, $z$, $w$ for complex variables, and $u$, $v$, $\alpha$, $\beta$ for real variables.
All constants $c_i$, $C_i$ etc are positive.
\end{remark}

\section{Point probabilities of a single branching process}\label{Smain}

In this section we study the point probabilities $\qN$ of the branching process $\bp=\bp_{Y,Z,Y^0,Z^0}$ from \refD{def:bp}. 
To formulate our main result we need some further definitions (which encapsulate fairly mild and natural conditions for the offspring distributions). 

\begin{definition}\label{def:kk}
Suppose that $R>1$, $\CM<\infty$, $k_1,k_2\in\bbN$ and $\delta>0$. 
  \begin{romenumerate}
  \item   
Let $\kko = \kko(R,\CM,\delta)$ be the set of probability distributions $\nu$ on $\bbN^2$
such that if $(Y,Z)\sim\nu$, then
\begin{align}
\label{RYZ}
 \E R^{Y+Z} & \le \CM, \\
\label{ey0}
\E Y & \ge\gd.  
\end{align}
\item 
Let 
 $\kki = \kki(k_1,k_2,\delta)$ be the set of probability distributions
 $\nu=(\pi_{i,j})_{i,j\ge0}$ on $\bbN^2$ 
such that 
\begin{equation}
  \label{a2}
\pi_{k_1,k_2}\ge\delta,\qquad 
\pi_{k_1+1,k_2}\ge\delta, \qquad
\pi_{k_1,k_2+1}\ge\delta.
\end{equation}
\item 
  Let $\kk = \kk(R,\CM,k_1,k_2,\delta):=\kko(R,\CM,\delta)\cap\kki(k_1,k_2,\delta)$.
  \end{romenumerate}
\end{definition}

We write $(Y,Z)\in\kko$ if the distribution of $(Y,Z)$ is in $\kko$, and
similarly for $\kki$ and $\kk$. 
The key condition here is the (uniform) bound \eqref{RYZ} on the 
probability generating functions. 
The condition~\eqref{a2} is needed, roughly speaking,
to ensure that $(Y,Z)$ is not essentially supported on a sublattice of $\bbN^2$. 
Note that $(Y,Z)\in \kki$ trivially implies
\begin{equation}\label{ez0}
 \E Z \ge \P(Z=k_2+1) \ge \gd,
\end{equation}
and similarly $\E Y \ge \gd$. 

The following theorem gives the qualitative behaviour of the size--$N$ point probabilities 
of the branching process $\bp=\bp_{Y,Z,Y^0,Z^0}$ from \refD{def:bp}. 
The statement of \refT{T1} is not self contained since the parameters $\Psi$, $\Phi$ and $\xx$ are defined 
(in a rather involved way) from the generating functions of $(Y,Z)$ and $(Y^0,Z^0)$, 
see \eqref{Psi}--\eqref{Phi} and \refL{LF} in Section~\ref{Ssum}. 
A key feature of the result is that the estimates and error-terms are uniform over all 
distributions $(Y^0,Z^0) \in \kko $ and $(Y,Z) \in \kk$, i.e., the explicit and implicit 
constants depend only on $R,\CM,k_1,k_2$ and $\delta$. 
Note that, from \eqref{xix} below, 
$\xi=0$ if and only if $\E Y = 1$, 
and that $\qN$ decays exponentially in~$\Theta(\eps^2 N)$
in the near-critical case $\E Y = 1 \pm \eps$.

\begin{theorem}[Point probabilities of $\bp$]\label{T1}
Suppose that $R>1$, $\CM<\infty$, $k_1,k_2 \in \bbN$, and $\delta>0$. 
Writing $\kko = \kko(R,\CM,\delta)$ and $\kk = \kk(R,\CM,k_1,k_2,\delta)$, 
there exists a constant $\ccname\ccT>0$ such that if
$(Y^0,Z^0)\in \kko$, $(Y,Z)\in\kk$, and $|\E Y-1|\le\ccT$, then 
for all $N \ge 1$ we have 
\begin{equation}\label{t1}
  \qN = N^{-3/2} e^{-N\xi}\bigpar{\theta+ O(N\qw)},
\end{equation}  
where, defining $\Psi$ and $\Phi$ as in \eqref{Psi}--\eqref{Phi} and $\xx$ as in \refL{LF}, we have 
\begin{gather}\label{xi}
  \xi=\xi_{Y,Z} :=-\Psi(\xx)\ge0, \\
\label{theta}
  \theta=\theta_{Y^0,Z^0,Y,Z} :=\sqrt{2\pi / |\Psi''(\xx)|}\ \Phi(\xx) = \Theta(1),
\end{gather}
and
\begin{equation}
\label{xix}
  \xi = \Theta(|\E Y-1|^2).
\end{equation}
Moreover, the implicit constants in~\eqref{t1}--\eqref{xix} depend only on $R,\CM,k_1,k_2$ and $\delta$. 
\end{theorem}

The remainder of this section is devoted to the proof of \refT{T1}. 
To this end we fix $R>1$, $\CM<\infty$, $k_1,k_2 \in \bbN$, and  $\delta>0$, 
and write $\kko = \kko(R,\CM,\delta)$ and $\kk = \kk(R,\CM,k_1,k_2,\delta)$ 
to avoid clutter. 
Let~$|\bp\xl|$ and~$|\bp\xs|$ denote the total numbers of type-$L$ and 
type-$S$ particles in $\bp$, so $|\bp|=|\bp\xl|+|\bp\xs|$, and set 
\begin{equation}\label{pab}
  \begin{split}
p_{\an,\bm}:=\P(|\bp\xl|=\an,\,|\bp\xs|=\bm).
  \end{split}
\end{equation}
Of course, $p_{\an,\bm}$ depends on the distributions of $(Y,Z)$ and $(Y^0,Z^0)$. 
In Section~\ref{Sint} we establish a simple integral formula for $p_{n,m}$. 
Then, in Section~\ref{Sintasymp} we use a version of the saddle point
method 
to estimate this integral asymptotically. 
Finally, in Section~\ref{Ssum} we prove~\eqref{t1} by summing all $p_{\an,\bm}$ with $n+m=N$.

\subsection{An integral formula for $p_{\an,\bm}$}\label{Sint}

In this section we derive an explicit integral formula for $p_{n,m}$, see~\eqref{pab-uv}. 
We start with a simple conditional version of the classical Otter--Dwass formula (see \eg{} \citet{Dwass}), 
which hinges on the random walk representation of a branching process and a well-known random-walk hitting time result.

\begin{lemma}\label{L1}
For all integers $\an\ge1$ and $\bm,\an_0,\bm_0\ge0$,
\begin{multline}\label{l1}
    \P\bigpar{|\bp\xl|=\an,\,|\bp\xs|=\bm \bigm| Y^0=\an_0,\,Z^0=\bm_0}
=
\frac{\an_0}{\an}
\P\biggpar{
\an_0+\sum_{1\le j\le \an}Y_j =\an,\ 
\bm_0+\sum_{1\le j\le \an}Z_j =\bm
} .   
\end{multline}
\end{lemma}
\begin{proof}
Let $(Y_j,Z_j)_{j\ge 1}$ be independent with each pair having the same distribution as $(Y,Z)$.
Since particles of type $S$ do not have any children,
by exploring the branching process~$\bp$ in the usual way (i.e., revealing the offspring of 
the particles of type $L$ one-by-one until none are left to explore), we have
\begin{multline*}
    \P\bigpar{|\bp\xl|=\an,\,|\bp\xs|=\bm \bigm| Y^0=\an_0,\,Z^0=\bm_0}
    \\
 =
 \P\biggpar{
 \an_0+\min_{0\le \an'<\an} \sum_{1\le j\le \an'}(Y_j-1) >0,\ 
 \an_0+\sum_{1\le j\le \an}(Y_j-1) =0,\ 
 \bm_0+\sum_{1\le j\le \an}Z_j =\bm } . 
\end{multline*}
That the \rhs{} of the above expression equals~\eqref{l1} is 
surely folklore (by conditioning on $\sum_{1\le j\le \an}Z_j=\bm-\bm_0$ 
this also follows directly from~\cite[Theorem~7]{Pitman:enum}); 
we include a short argument.
Namely, by a version of the well-known Cyclic Lemma (sometimes also 
called Spitzer's combinatorial lemma), 
see, e.g.,~\cite[Lemma~15.3]{SJ264} or \cite[Lemma~6.1]{Pitman2006}, 
for any sequence $(y_1, \ldots, y_n) $ with $y_i \in \{-1,0,1,2,\ldots\}$ 
and $n_0 + \sum_{1 \le i \le n} y_i = 0$, there are exactly~$n_0$ cyclic 
shifts of $(y_1, \ldots, y_n)$ for which all corresponding partial sums 
$s_i = y_1 + \cdots + y_i$ of length $i \le n-1$ satisfy $n_0 + s_i > 0$. 
Hence, taking a uniformly random cyclic shift of the~$n$ independent 
variables $(Y_j-1,Z_j)$, the formula~\eqref{l1} follows. 
\end{proof}

\begin{remark}
This two-type version of the Otter--Dwass formula is a simple
variation of the usual one-type case;
this is because one type is barren and can essentially be ignored.
For a much more complicated formula in
the general multi-type case, see
\citet{ChaumontLiu}.
\end{remark}

The probability on the \rhs{} of \eqref{l1} can be expressed using
generating functions as
\begin{equation}
  [y^{\an-\an_0}z^{\bm-\bm_0}]\bigpar{g(y,z)^\an}
=  [y^{\an}z^{\bm}]\bigpar{y^{\an_0}z^{\bm_0}g(y,z)^\an}.
\end{equation}
For $\an\ge1$ and $\bm\ge0$, recalling the notation \eqref{pab}
and summing \eqref{l1} over all $\an_0,\bm_0$, we thus obtain
\begin{equation}\label{pab-g}
  \begin{split}
\an p_{\an,\bm} 
&=\sum_{\an_0,\bm_0\ge0} \P(Y^0=\an_0,Z^0=\bm_0)\an_0 
[y^{\an}z^{\bm}]\bigpar{y^{\an_0}z^{\bm_0}g(y,z)^\an}
\\&
=[y^{\an}z^{\bm}]\bigpar{\tg_0(y,z)g(y,z)^\an},
  \end{split}
\end{equation}
where
\begin{equation*}
  \begin{split}
    \tg_0(y,z):=\sum_{\an_0,\bm_0\ge0} \P(Y^0=\an_0,Z^0=\bm_0)\an_0 y^{\an_0}z^{\bm_0}
=y \pddd y g_{Y^0,Z^0}(y,z). 
  \end{split}
\end{equation*}

For later use, we also define
\begin{equation}\label{tf0}
  \tf_0(y,z):=\tg_0(e^y,e^z)=\pddd y f_{Y^0,Z^0}(y,z)
=\E \bigpar{Y^0e^{yY^0+zZ^0}}.
\end{equation}
\begin{remark}
Let $G(y,z):=\E\bigpar{ y^{|\bp\xl|}z^{|\bp\xs|}}$ be the bivariate generating
function for the size of the branching process $\bp$, and let
$G_1(y,z):=\E\bigpar{ y^{|\bpil|}z^{|\bpis|}}$ be the corresponding
generating function when starting with a single particle of type $L$.
Then 
$G(y,z)=g_0(G_1(y,z),z)$ and
$G_1(y,z)=yg(G_1(y,z),z)$, and
the formula~\eqref{pab-g} can alternatively be obtained 
by the Lagrange inversion formula in the B\"urmann form, see \eg{}
\cite[A.(14)]{Flajolet}, regarding the generating functions as (formal) 
power series in $y$ with coefficients that are power series in $z$. 
We omit the details.
\end{remark}

The extraction of coefficients in \eqref{pab-g} can be performed by complex
integration in the usual way 
(e.g., using Cauchy's integral formula to evaluate 
$\frac{\partial^{\an+\bm}}{\partial y^\an\partial z^\bm}\bigl(\tg_0(y,z)g(y,z)^\an\bigr)\big|_{y=z=0}=\an!\,\bm!\, n p_{\an,\bm}$
as in the textbook proof of Cauchy's estimates),
yielding the formula  
\begin{equation*}
  \begin{split}
    \an p_{\an,\bm}=\frac{1}{(2\pi\ii)^2}\oint \oint
    y^{-\an}z^{-\bm}\tg_0(y,z)g(y,z)^\an\frac{\dd y}{y}\frac{\dd z}{z},
  \end{split}
\end{equation*}
where we integrate (for example) over two circles with centre~$0$ and
radii such that $\tg_0(y,z)$ and $g(y,z)$ are defined. In particular, if $(Y,Z)$ and $(Y^0,Z^0)$
are both in $\cK^0$, then for any
$\ga,\gb<\log R$ we can integrate over $|y|=e^{\ga}$ and $|z|=e^\gb$, and
the standard change of variables $y=e^{\ga+\ii u}$, $z=e^{\gb+\ii v}$ then yields
\begin{equation}\label{pab-uv}
    \an p_{\an,\bm} =\frac{1}{4\pi^2}\intpipi \intpipi
e^{-\an(\ga+\ii u)-\bm(\gb +\ii v)}
\tf_0(\ga+\ii u,\gb +\ii v)f(\ga+\ii u,\gb +\ii v)^\an\dd u\dd v.
\end{equation}

\begin{remark}
  Alternatively, \eqref{pab-uv} can be obtained from \eqref{l1} by first
considering suitably tilted versions of the random variables 
(\cf{} \citet{Cramer}), and then passing to \chf{s} and making a Fourier
inversion. 
\end{remark}

\begin{remark}
It is not hard to write an integral formula for the final probability
$p_N=\sum_{m+n=N}p_{n,m}$ that we are aiming to estimate. 
For example, multiplying \eqref{pab-g} by $x^n/n$ and summing
we see that $p_{n,m}=[x^ny^nz^m]H(x,y,z)$, where $H(x,y,z)=-\tg_0(y,z)\log(1-xg(y,z))$.
Thus one can find $p_N$ by extracting the coefficient of $w^0t^N$ in $H(w,t/w,t)$. However, the corresponding integral
does not obviously lend itself to asymptotic evaluation by methods such as those used here.
Still, a direct estimate of $p_N$ may
perhaps be possible by appropriate singularity analysis.
\end{remark}

\subsection{An asymptotic estimate of $p_{n,m}$}\label{Sintasymp}

In this section we estimate the integral~\eqref{pab-uv} asymptotically (see Theorem~\ref{Tpab} below), 
using parameters defined in terms of the moment generating function $f(y,z)=f_{Y,Z}(y,z)=\E\bigpar{e^{yY+zZ}}$. 
Whenever $f$ is defined and non-zero, let
\begin{equation}\label{gf}
  \gf(y,z) = \gf_{Y,Z}(y,z) :=\log f_{Y,Z}(y,z) = \log f(y,z),
\end{equation}
taking the principal value of the logarithm; we shall only consider $\gf$ on domains on which $|f-1|\le 1/2$.
The next lemma simply states that in suitable domains, $f$, $\gf$ and their (partial) derivatives
are all bounded.

\begin{lemma}\label{LB}
There exist constants $0<\ccname\cca \le (\log R)/2$ and $\CCname\CCa\xm$,
$m \in \bbN$,
such that if $(Y,Z)\in\kko$ and $m \in \bbN$, then the following hold.
\begin{romenumerate}
\item \label{LBa}
If $\ga,\gb,u,v\in \bbR$ with $|\ga|,|\gb|\le\cca$, then
$\norm{D^mf(\ga+\ii u,\gb+\ii v)}\le\CCa\xm$.
\item \label{LBc}
If, in addition, $|u|,|v|\le\cca$, then
$\gf(\gabuv)$ is defined, and
$\norm{D^m\gf(\ga+\ii u,\gb+\ii v)}\le\CCa\xm$.
\item \label{LBx}
If $|\ga|,|\gb|\le\cca$, then $\pddd{y}f(\alpha,\beta)\ge \delta/2$.
\end{romenumerate}
\end{lemma}

\begin{proof}
  \pfitemref{LBa}
When $|y|, |z|\le R$, then $|g(y,z)|=|\E(y^Yz^Z)|\le \E(|y|^Y|z|^Z) \le \E R^{Y+Z}$,
which is at most $\CM$ by assumption. Thus
$|f(y,z)|\le \CM$ when $\Re(y)\le\log R$ and $\Re(z)\le\log R$.
Recall that $R>1$ by assumption, so $\log R>0$. For any $\cca \le (\log R)/2$, say, 
for suitable $\CCa\xm>0$ statement \ref{LBa} follows by standard Cauchy estimates 

\pfitemref{LBc}
Let $C=\CCa^{(1)}$ denote the constant from the above proof of~\ref{LBa}. Set $\cca:=\min\set{(\log R)/2,1/(8C)}$.
Since $f(0,0)=g(1,1)=1$, it follows from~\ref{LBa} that 
if $|\ga|,|\gb|,|u|,|v|\le \cca$, then
\begin{equation}
  \bigabs{f(\ga+\ii u,\gb+\ii v)-1}
\le (|\ga+\ii u|+|\gb+\ii v|)C\le 4 \cca C \le 1/2 ,
\end{equation}
so $\gf(\gabuv)$ is defined and bounded. 
Furthermore, after decreasing $\cca$ and increasing $\CCa\xm$, if necessary, 
the bounds for the derivatives now again follow by Cauchy's estimates. 

\pfitemref{LBx}
Let $\tf=\pddd{y}f$. By our assumption \eqref{ey0}, $\tf(0,0)=\E Y\ge\gd$.
Furthermore, $D\tf(\ga,\gb)=DD_1f(\ga,\gb)=O(1)$ for $|\ga|,|\gb|\le\cca$ by
part \ref{LBa}. Consequently, after reducing $\cca$ if necessary,
we have 
$\tf(\ga,\gb)\ge\frac12\gd$ for~$|\ga|$,~$|\gb|\le\cca$. 
\end{proof}

The next lemma expresses, in a quantitative form, the unsurprising fact that if we evaluate
the probability generating function $g(y,z)=g_{Y,Z}(y,z)=\E (y^Yz^Z)$ at $y$, $z$ which are not positive real
numbers, then there is significant cancellation, i.e., $|g(y,z)|$ is significantly smaller
than $g(|y|,|z|)$. It will be more convenient to write this in terms of the moment generating function $f=f_{Y,Z}$ rather than $g$.

\begin{lemma}
  \label{LC}
There exists a constant $\ccname\ccd>0$ such that
if $(Y,Z)\in\kk$ 
and $\ga,\gb,u,v \in \bbR$ with 
$|\ga|,|\gb|\le\cca$ and $|u|,|v|\le\pi$, then
\begin{equation}
  \label{lc}
\bigabs{f(\gabuv)} \le f(\ga,\gb)e^{-\ccd(u^2+v^2)}.
\end{equation}
\end{lemma}
\begin{proof}
  Let $\pi_{k,l}:=\P(Y=k,Z=l)$. Then 
  \begin{equation*}
    f(\gabuv)=\sum_{k,l\ge0}\pi_{k,l}e^{k(\ga+\ii u)+l(\gb+\ii v)} ,
  \end{equation*}
and thus $f(\ga,\gb) >0$. 
Then 
\begin{equation*}
f(\ga,\gb)^2-|f(\gabuv)|^2
=\sum_{k,l,m,n}\pi_{k,l}\pi_{m,n}e^{(k+m)\ga+(l+n)\gb}\Bigpar{1-\Re e^{\ii(k-m) u+\ii(l-n)v}}  .
\end{equation*}
Each term on the \rhs{} is non-negative, and considering just the cases
$(k,l,m,n)=(k_1,k_2,k_1+1,k_2)$ and $(k_1,k_2,k_1,k_2+1)$,
recalling \eqref{a2} we obtain 
\begin{equation*}
\begin{split}
f(\ga,\gb)^2-|f(\gabuv)|^2 
& \ge
\delta^2e^{(2k_1+1)\ga+2k_2\gb}(1-\cos u)
+ \delta^2e^{2k_1\ga+(2k_2+1)\gb}(1-\cos v)\\
& = \Omega(u^2+v^2),
\end{split}
\end{equation*}
since $1-\cos x=\Omega(x^2)$ for $|x| \le \pi$.
Moreover, by \refL{LB}\ref{LBa}, $f(\ga,\gb)=O(1)$.
Consequently 
\begin{equation*}
1-  |f(\gabuv)|^2/f(\ga,\gb)^2 \ge 2\ccd(u^2+v^2)
\end{equation*}
for some constant $\ccd>0$, and thus 
\begin{equation*}
  |f(\gabuv)|^2/f(\ga,\gb)^2 \le 1-2\ccd(u^2+v^2) \le e^{-2\ccd(u^2+v^2)},
\end{equation*}
establishing~\eqref{lc} since $f(\ga,\gb) > 0$. 
\end{proof}

We next establish that the symmetric bilinear form $D^2\gf(\ga,\gb)$ is positive-definite; 
a variant of the lower bound~\eqref{ld} could also be proved by first considering $D^2\gf(0,0)$
and then using continuity. For the interpretation of $D^2\gf\xgagb[u,v]$, see~\eqref{Dmf}. 

\begin{lemma}\label{LD}
If $(Y,Z)\in\kk$ and $\ga,\gb\in\bbR$ with 
$|\ga|,|\gb|\le\cca$, then
$D^2\gf(\ga,\gb)\ge \ccd I$, \ie,
\begin{equation}
  \label{ld}
D^2\gf\xgagb[u,v]\ge\ccd (u^2+v^2), \qquad u,v\in\bbR.
\end{equation}
In particular,
$\Det(D^2\gf(\ga,\gb))\ge\ccd^2$.
\end{lemma}
\begin{proof}
We first consider only $|u|,|v|\le\cca$, so \refL{LB}\ref{LBc} applies.
  Then the estimate \eqref{lc} can be written
  \begin{equation}\label{balle}
\Re\gf(\gabuv)\le \gf(\ga,\gb)-\ccd(u^2+v^2).
  \end{equation}
A Taylor expansion yields
  \begin{equation*}
\gf(\gabuv)= \gf(\ga,\gb)+\ii D\gf\xgagb\xxuv
-\tfrac12D^2\gf\xgagb\xxuv+O\bigpar{(|u|+|v|)^3}.
  \end{equation*}
Since $\gf(\ga,\gb)$ is real for real $\ga$ and $\gb$, all derivatives
$D^m\gf(\ga,\gb)$ are real. Hence, when taking the real part, the linear
term vanishes, and \eqref{balle} implies 
\begin{equation*}
  \tfrac12D^2\gf\xgagb\xxuv\ge \ccd(u^2+v^2)+O\bigpar{(|u|+|v|)^3}.
\end{equation*}
Exploiting bilinearity, by replacing $(u,v)$ with $(tu,tv)$ and  
letting $t\to0$, we now obtain \eqref{ld} for all $u,v\in\bbR$, 
with room to spare. 

Finally, by~\eqref{rM}, note that \eqref{ld} can be written 
$D^2\gf(\ga,\gb)\ge \ccd I$. 
This says that both eigenvalues are~$\ge\ccd$, and
thus the determinant is~$\ge\ccd^2$.
\end{proof}

For $|\ga|,|\gb|\le\cca$, define  
\begin{equation}
  \label{psi}
\psi(\ga,\gb):=\gf(\ga,\gb)-\ga D_1\gf(\ga,\gb)-\gb D_2 \gf(\ga,\gb).
\end{equation}
We are now ready to estimate the integral~\eqref{pab-uv} for $p_{\an,\bm}$ 
using a (two-dimensional) version of the saddle point method (see,
e.g.,~\cite[Chapter~VIII]{Flajolet}).
We defer the problem of finding suitable $(\alpha,\beta)$ 
satisfying equation~\eqref{tpab0} to Section~\ref{Ssum}. 
Recall that $\tf_0(y,z) = \pddd y f_{Y^0,Z^0}(y,z) =\E \bigpar{Y^0e^{yY^0+zZ^0}}$, 
see~\eqref{tf0}. 

\begin{theorem}\label{Tpab}
Suppose that $(Y^0,Z^0)\in\kko$ and $(Y,Z)\in \kk$.
Suppose further that $n\ge1$, $m\ge0$ are integers and that $\ga$, $\gb$ are real numbers
with $|\ga|,|\gb|\le\cca$ such that
\begin{equation}
  \label{tpab0}
D\gf(\ga,\gb)=(1,m/n).
\end{equation}
Then
\begin{equation}\label{tpab}
p_{n,m}=  
n^{-2}e^{n\psi(\ga,\gb)}
\Bigpar{(2\pi)^{-1} \tf_0(\ga,\gb) \Det\bigpar{ D^2\gf(\ga,\gb)}\qqw
+O(n^{-1})} ,
\end{equation}
where the implicit constant depends only on the parameters 
$R,\CM,k_1,k_2,\delta$ of~$\kko$ and~$\kk$. 
\end{theorem}
\begin{proof}
  We write \eqref{pab-uv} as
\begin{equation}\label{pnm}
  p_{n,m}=\frac{1}{4\pi^2n}e^{-n\ga-m\gb}f(\ga,\gb)^n I_1,
\end{equation}
where
\begin{equation}\label{imn}
  I_1:=
\intpipi \intpipi
e^{-\ii\an u-\ii\bm v}
\tf_0(\ga+\ii u,\gb +\ii v)\parfrac{f(\ga+\ii u,\gb +\ii v)}{f(\ga,\gb)}^\an
 \dd u\dd v.
\end{equation}
Using assumption~\eqref{tpab0} we have $\psi(\ga,\gb) = \gf(\ga,\gb)-\ga-\gb m/n$, so
\begin{equation}\label{jepp}
e^{-n\ga-m\gb}f(\ga,\gb)^n=
e^{-n\ga-m\gb+n\gf(\ga,\gb)}
=e^{n\psi(\ga,\gb)}.
\end{equation}
We shall estimate~\eqref{imn} using Laplace's method (in two dimensions),
cf.\ e.g.\ \cite[Appendix B.6]{Flajolet}. 
Roughly speaking, the idea is as follows. 
We view the integrand as a product of a term independent of $n$ with a term that is exponential
in~$n$. As we shall see, the condition~\eqref{tpab0} ensures that the
exponent has a stationary point, 
in fact a maximum, at~$u=v=0$. 
It turns out that the main contribution is near to this point, 
and here the exponent may
be approximated by a quadratic, 
leading to a (two-dimensional) Gaussian integral. 

Applying \refL{LB}\ref{LBa} to $(Y^0,Z^0)$ shows that $\tf_0(\ga,\gb)=O(1)$.
Since $\Det\bigpar{ D^2\gf(\ga,\gb)}=\Omega(1)$ by Lemma~\ref{LD}, and 
$\psi(\ga,\gb)=O(1)$ by \eqref{psi} and \refL{LB}\ref{LBc}, 
the conclusion~\eqref{tpab} holds for any fixed~$n$ simply by taking the implicit
constant large enough. Thus we may assume that~$n$ is at least any given constant~$n_0$,
and in particular that $n^{-0.4}\le \cca$.

Applying \refL{LB}\ref{LBa} to $(Y^0,Z^0)$ also shows that
$\tf_0(\gabuv)=O(1)$. 
Hence,
if $|u|\ge \nx$ or $|v|\ge\nx$, then by \refL{LC} the integrand 
in \eqref{imn} is 
$O\bigpar{e^{-\ccd n\cdot n^{-0.8}}}=O\bigpar{e^{-\ccd n^{0.2}}}=O\bigpar{n^{-99}}$.
On the other hand, if $|u|$, $|v|\le\nx$ then, since $\nx\le\cca$,
\refL{LB}\ref{LBc} shows that $\gf(\gabuv)$ is defined and we obtain
\begin{equation}
  \label{i12}
  I_1=O\bigpar{n^{-99}}+I_2 ,
\end{equation}
 with
\begin{equation}\label{qq}
I_2:=
\int_{-\nx}^{\nx}\int_{-\nx}^{\nx}
\tf_0(\ga+\ii u,\gb +\ii v)
e^{\an\xpar{\gf(\ga+\ii u,\gb +\ii v)-\gf(\ga,\gb)}-\ii\an u-\ii\bm v}
 \dd u\dd v.
\end{equation}
Considering a Taylor expansion of $\gf$ around $(\alpha,\beta)$, 
and noting that the linear terms cancel by our assumption~\eqref{tpab0}, we have
\begin{multline}\label{cf}
  \an\xpar{\gf(\ga+\ii u,\gb +\ii v)-\gf(\ga,\gb)}-\ii\an u-\ii\bm v
\\
= -n\tfrac12 D^2\gf\xgagb\xxuv 
-n\tfrac{\ii}6 D^3\gf\xgagb\xxuv +O\bigpar{n(|u|+|v|)^4},
\end{multline}
where we used \refL{LB}\ref{LBc} to bound the error term.
For $|u|$, $|v|\le\nx$, note that \refL{LB}\ref{LBc} implies 
$nD^3\gf(\ga,\gb)\xxuv =O(n(|u|+|v|)^3)=O(n^{-0.2})=O(1)$, 
and $O(n(|u|+|v|)^4)=O(n^{-0.6})=O(1)$.
Hence, writing for brevity 
\begin{equation*}
Q := D^2\gf(\ga,\gb) ,
\end{equation*}
the exponential factor in \eqref{qq} is
{\multlinegap=0pt
\begin{multline}\label{pn}
  e^{-\tfrac12 nQ[u,v]}
\exp\Bigpar{-n\tfrac{\ii}6 D^3\gf\xgagb\xxuv +O\bigpar{n(|u|+|v|)^4}}
\\
=e^{-\tfrac12 nQ[u,v]}\Bigpar{
  1-n\tfrac{\ii}6 D^3\gf\xgagb\xxuv 
+O\bigpar{n(u^4+v^4)}+O\bigpar{n^2(u^6+v^6)}}.
\end{multline}}%
Recalling $\tf_0=\pddd{y}f_{Y^0,Z^0}$, using~\refL{LB}\ref{LBa} 
we also have the Taylor expansion
\begin{equation}\label{qn}
  \tf_0(\gabuv)=
\tf_0(\ga,\gb)+\ii D\tf_0\xgagb\xxuv + O\bigpar{(|u|+|v|)^2}.
\end{equation}
Multiplying together \eqref{pn} and \eqref{qn}, the integrand in
\eqref{qq} is thus 
{\multlinegap=0pt
\begin{multline}\label{ollo}
e^{-\tfrac12 nQ[u,v]}
\Bigl(\tf_0(\ga,\gb)+\ii D\tf_0{}\xgagb\xxuv 
-n\tf_0(\ga,\gb)\tfrac{\ii}6 D^3\gf\xgagb\xxuv 
\\
+ O\bigpar{u^2+v^2} 
+O\bigpar{n(u^4+v^4)}+O\bigpar{n^2(u^6+v^6)}
\Bigr).
\end{multline}}%
When we integrate, the terms with $D\tf_0$ and $D^3\gf$ are odd functions of
$(u,v)$ so their integrals vanish. Hence,
\begin{equation*}
  \begin{split}
I_2
&=\int_{-\nx}^{\nx}\int_{-\nx}^{\nx}
e^{-\tfrac12 nQ[u,v]}
\Bigl(\tf_0(\ga,\gb)
+ O\bigpar{u^2+v^2} 
+O\bigpar{n(u^4+v^4)}+O\bigpar{n^2(u^6+v^6)}
\Bigr)\dd u\dd v .
\end{split}
\end{equation*}
Recalling that $Q=D^2\gf(\ga,\gb)$, by~\refL{LD} we have $Q[u,v]=\Omega(u^2+v^2)$.  
Since for $k \in \{1,2,3\}$ we have 
$\iint_{\bbR^2} e^{-a(u^2+v^2)}(u^2+v^2)^k\dd u\dd v=O(a^{-(k+1)})$, 
it follows that 
\begin{equation*}
I_2 =\tf_0(\ga,\gb)\int_{-\nx}^{\nx}\int_{-\nx}^{\nx}
e^{-\tfrac12 nQ[u,v]} \dd u\dd v
+O(n^{-2}).    
\end{equation*}
Since $Q=D^2\gf(\ga,\gb)$ is symmetric and positive-definite by Lemma~\ref{LD}, 
we have the following standard Gaussian integral over $\bbR^2$:
\begin{equation}\label{eq:GI}
\iint_{\bbR^2} e^{-nQ[u,v]/2} \dd u\dd v
=n^{-1} \cdot 2\pi (\Det(Q))\qqw .
\end{equation}
Since $Q[u,v]=\Omega(u^2+v^2)$, the contribution 
of the range $\max\{|u|,|v|\}\ge n^{-0.4}$ to the 
above integral~\eqref{eq:GI} is again exponentially small. 
Hence 
\begin{equation}\label{kn}
I_2
=\tf_0(\ga,\gb) \cdot n\qw 2\pi (\Det(Q))\qqw
+O(n^{-2}).    
\end{equation}
The result follows by combining \eqref{pnm}, \eqref{jepp},
\eqref{i12} and~\eqref{kn}.
\end{proof}

We next estimate the exponent in~\eqref{tpab}, without assuming that equation~\eqref{tpab0} holds.

\begin{lemma}\label{Lpsi}
There exists a constant $0<\ccname\cckk\le\cca$
such that
if $(Y,Z)\in\kk$ and $\ga,\gb \in \bbR$ with 
$|\ga|,|\gb|\le\cckk$, then
\begin{equation}
  \label{lpsi}
  \psi(\ga,\gb)\le -\tfrac{1}{4}\ccd (\ga^2+\gb^2).
\end{equation}
Moreover, $\psi(0,0)=0$, $D\psi(0,0)=0$ and $D^2\psi(0,0)\le -\ccd I$. 
\end{lemma}

\begin{proof}
We have $\psi(0,0)=\gf(0,0)=0$.
Furthermore,
differentiating \eqref{psi} yields
\begin{align*}
D_1\psi(\ga,\gb)& = -\ga D_{11}\gf(\ga,\gb) -\gb D_{12}\gf(\ga,\gb),
\\
D_2\psi(\ga,\gb)& = -\ga D_{12}\gf(\ga,\gb) -\gb D_{22}\gf(\ga,\gb),
\end{align*}
and thus $D\psi(0,0)=(0,0)$.
Differentiating again shows that
$D_{ij}\psi(0,0)=-D_{ij}\gf(0,0)$ for  
all $i,j\in\set{1,2}$. Hence, using \refL{LD},
\begin{equation*}
  D^2\psi(0,0)=-D^2\gf(0,0)\le -\ccd I.
\end{equation*}
Moreover, it follows from \refL{LB}\ref{LBc} that $\norm{D^3\psi(\ga,\gb)}=O(1)$
for $|\ga|,|\gb|\le\cca$.
Consequently, a Taylor expansion yields \eqref{lpsi} 
for $\cckk$ sufficiently small.
\end{proof}

\subsection{Summing $p_{n,m}$: proof of \refT{T1}}\label{Ssum}

In this section we prove \refT{T1} by summing several different estimates of the point probabilities in 
\begin{equation}\label{qN}
 \qN=\sum_{n=0}^N p_{n,N-n} .
\end{equation}
Throughout we consider, as in \eqref{tpab}, only real inputs $\alpha$, $\beta$ for the 
various functions $f$, $\gf$ etc. Thus, all relevant functions are treated as mapping from (subdomains in)
$\bbR^n$ to $\bbR^m$ for suitable $n$, $m$.

An individual of type $L$ has on average $\E Y$ children of type $L$ and $\E Z$ children of type $S$. 
So, in the near-critical case $\E Y \approx 1$, we expect that the overall fraction 
of type $L$ individuals in~$\bp$ should be close to 
\begin{equation*}
 x_0:=1/(1+\E Z).
\end{equation*}
This suggests that the contribution from terms in \eqref{qN} with~$n/N$ 
far from~$x_0$ will be negligible, and we shall later confirm this by standard Chernoff-like estimates. 
Below our main focus is thus on the terms where~$n/N$ is close to~$x_0$. 
Here the plan is to rewrite the asymptotic estimate~\eqref{tpab} 
for $p_{n,N-n}$ using the following version of the inverse function theorem, 
where we explicitly state uniformity for a set of functions. We define
\[
 \Bd_r:=\set{x\in\bbR^\Nd:|x|<r} \qquad \text{ and } \qquad B_r := \Bt_r = \set{x\in\bbR^2:|x|<r}.
\]

\begin{lemma}[Inverse function theorem]\label{LI}
Let $\Nd\ge1$ be an integer and $r>0$ a real number. 
For every~$0 < A<\infty$, there exist $\rhox > 0$ and $0 < r_1< r$, both depending only on $A,r$,
such that if 
$F:\Bd_r\to\bbR^\Nd$ is twice continuously
differentiable and satisfies
\begin{romenumerate}
\item $F(0)=0$, 
\item 
$DF(0)$ is invertible and $\norm{DF(0)\qw}\le A$, and 
\item 
$\norm{D^2F(x)}\le A$ for all $x\in \Bd_r$,
\end{romenumerate}
then there exists a twice continuously
differentiable function $G:\Bd_\rhox\to \Bd_{r}$ with 
$G(0)=0$ and $F(G(y))=y$ for $y\in \Bd_\rhox$. 
Furthermore, for each $y\in \Bd_\rhox$, $x=G(y)$ is the unique $x\in \bbR^d$ with $|x| \le r_1$ such that $F(x)=y$.
Moreover, $\norm{DG(y)}=O(1)$ and $\norm{D^2G(y)}=O(1)$, uniformly for $y\in \Bd_\rhox$ and all such $F$, 
and if $F$ is infinitely differentiable or (real) analytic, then so is $G$. 
\end{lemma}

\begin{proof}
This follows by a standard proof of the inverse function theorem; we
give some details for completeness. 

First, let $r_1:=\frac12\min\xcpar{r,A\qww}$. If $|x|\le r_1$, then by the
mean-value theorem
$\norm{DF(x)-DF(0)}\le A|x|\le Ar_1$. 
Hence,
$\norm{DF(0)\qw DF(x)-I}\le  A^2r_1\le\frac12$, and thus
$DF(0)\qw DF(x)$ is invertible and its inverse
has norm at most~$2$ 
(e.g., by the von Neumann series representation of the inverse).
Consequently,
$DF(x)$ is invertible and
\begin{equation}\label{DF}
 \norm{DF(x)\qw}\le \norm{DF(0)\qw} \cdot  
\norm{(DF(0)\qw DF(x))\qw}\le 2A
\hbox{\qquad for }|x|\le r_1.
\end{equation}
Next, let $\rhox:=r_1/(2A)$.
If $|y|<\rhox$, define inductively $x_0:=0$ and
$x_{n+1}:=\Gamma(x_n)$, where
\[
  \Gamma(x):=x+DF(0)\qw(y-F(x)).
\]
Using $\norm{DF(0)\qw}\le A$ and that
$\norm{D \Gamma(x)} = \norm{I-DF(0)\qw DF(x)}\le \frac12$ if $|x| \le r_1$,
it is easy to show by induction that
$|x_n|\le(1-2^{-n})r_1$
and
$|x_{n+1}-x_n| \le 2^{-n}A \rhox\le 2^{-n-1}r_1$.
Hence $x_n$ is defined for all $n\ge0$, and converges to some
$x$ with $|x|\le r_1<r$.
Furthermore, $y-F(x_n)=DF(0)(x_{n+1}-x_n)\to0$ as \ntoo,
and thus by continuity $F(x)=y$. Define $G(y):=x$.

This shows that the inverse function $G$ exists in~$\Bd_\rhox$. 
The uniqueness statement is immediate, since any~$x \in \bbR^d$ satisfying
$F(x)=y$ is a fixed point of $\Gamma(x)$, which is a contraction for $|x| \le r_1$. 
Differentiability (and analyticity when $F$ is analytic) 
follows in the usual way (or by appealing to a
standard version of the inverse function theorem, locally at~$G(y)$).
Finally, $DG(y)=DF(x)\qw$, and thus $\norm{DG(y)}\le 2A$ by \eqref{DF}.
Another differentiation (using the chain rule) then yields $\norm{D^2G(y)}=O(1)$.
\end{proof}

Our next aim is to construct an (implicit) solution 
$(\alpha,\beta)=h(n/N)$ to equation~\eqref{tpab0} 
when $N=n+m$ and~$n/N$ is close to~$x_0=1/(1+\E Z)$. 
We start by applying \refL{LI} to the function $F:B_{\cckk}\to\bbR^2$ 
defined~by
\begin{equation}\label{F}
  F(\ga,\gb):=\Bigpar{\E Y - D_1\gf(\ga,\gb),\;
    \frac{1}{1+D_2\gf(\ga,\gb)}-x_0}.
\end{equation}
Note that $D_2\gf(\ga,\gb)=D_2f(\ga,\gb)/f(\ga,\gb)\ge0$, and thus
$F(\ga,\gb)$ is well-defined.
Furthermore, 
$D\gf(0,0)=(\E Y,\E Z)$, and thus $F(0,0)=(0,0)$.
Moreover, using matrix form (where the first column 
is $\pddd{\ga}$ of the vector
valued function $F$ and the second is $\pddd{\gb}$), we have 
\begin{equation}\label{br}
  D F=
  \begin{pmatrix}
    -1&0\\
0& -(1+D_2\gf)\qww
  \end{pmatrix}
D^2\gf.
\end{equation}
It follows from \refL{LD} that $\bignorm{(D^2\gf(\ga,\gb))\qw} =O(1)$,
and then \eqref{br} together with \refL{LB} yields 
\begin{equation}\label{DFi}
 \norm{DF(\ga,\gb)\qw} =O(1).
\end{equation}
\refL{LB} also implies 
$\norm{D^2F(\ga,\gb)} =O(1)$. 
Consequently, \refL{LI} applies (with $d=2$) and yields a constant 
$\rhox=\ccname\ccq>0$
and a function $G:B_{\ccq}\to B_{\cckk}$ such that 
\begin{equation}
  \label{FG}
F(G(y))=y
\hbox{\qquad for }y\in B_{\ccq}.
\end{equation}
Recall that $x_0=1/(1+\E Z)$. 
Since $\E Z=D_2 f(0,0)=O(1)$ by \refL{LB}
and $\E Z\ge \gd>0$ by \eqref{ez0},
there exists a constant $c>0$ such that $c\le x_0\le1-c$.
Let 
\[
\ccname\ccy:=\tfrac{1}{2}\min\{\ccq,\:c\}.
\] 
Suppose that $|\E Y-1|<\ccy$.
If also $\abs{x-x_0}<\ccy$, 
then $\bigpar{\E Y-1,x-x_0}\in B_{\ccq}$; we then
define
\begin{equation}\label{h}
  h(x):=G\xpar{\E Y-1,x-x_0}\in B_{\cckk}.
\end{equation}
Furthermore, $\abs{x-x_0}<\ccy\le \frac12c\le \frac12x_0$ implies 
$x\ge\frac12x_0\ge\ccy$
and $1-x\ge\ccy$.
Now suppose that $0< n\le N$ and that
 $\abs{n/N-x_0}<\ccy$,
and let 
$m:=N-n$ and 
$(\ga,\gb):=h\bigpar{n/N}$. 
Then, by \eqref{h} and \eqref{FG}, 
\begin{equation}
F(\ga,\gb)=
\Bigpar{\E Y-1,\;\frac{n}N-x_0}  
=\Bigpar{\E Y-1,\;\frac{1}{1+m/n}-x_0} .
\end{equation}
Definition \eqref{F} shows that \eqref{tpab0} holds. 
Hence, by \refT{Tpab}, \eqref{tpab} holds.
For $\abs{x-x_0}<\ccy$ define
\begin{align}
  \Psi(x)&:= x\psi(h(x)),\label{Psi}
\\
  \Phi(x)&:= 
  (2\pi)^{-1}x^{-2} \tf_0(h(x)) \Det\bigpar{ D^2\gf(h(x))}\qqw.\label{Phi}
\end{align}
Recall that $h(x)\in B_{\cckk}\subseteq B_{\cca}$, and note that
\refL{LD} implies $\Det\bigpar{ D^2\gf(h(x))}\ge\ccd^2$; 
thus~$\Psi(x)$ and~$\Phi(x)$ are well-defined. 
Then,  
still assuming
$|\E Y-1|<\ccy$,
 $\abs{n/N-x_0}<\ccy$
and $(\ga,\gb):=h\bigpar{n/N}$, 
we see that~\eqref{tpab} can be written
\begin{equation}\label{tpnN}
p_{n,N-n}=  
N^{-2}e^{N\Psi(n/N)}
\bigpar{\Phi(n/N)+O(N^{-1})}.
\end{equation}
(Here, we use $\abs{n/N-x_0}<\ccy \le x_0/2$ to bound $n\ge \ccy N$, so an $O(n^{-1})$ error term is $O(N^{-1})$.)

We next show that, in the relevant 
domains, the functions $\Phi$, $\Psi$ and their (partial) 
derivatives are all bounded. 

\begin{lemma}\label{LE}
For each $m\ge0$,
there exists a constant $\CCname\CCf\xm$ such that   
  if $|\E Y-1|<\ccy$ and  $\abs{x-x_0}<\ccy$, then
$\norm{D^m\Phi(x)}\le\CCf\xm$ and
$\norm{D^m\Psi(x)}\le\CCf\xm$.
\end{lemma}

\begin{proof}
We saw in the proof of \refL{LI} that $DG(y)=\bigpar{DF(G(y)}\qw$, which is
bounded for $y\in B_{\ccq}$ by~\eqref{DF}. By further differentiations, using
the chain rule, \refL{LB}\ref{LBc} and induction, it follows that
for each~$m\ge0$, 
\begin{equation}\label{DmG}
\norm{D^m  G\xpar{\E Y-1,x-x_0}}=O(1)
\end{equation}
when $|\E Y-1|<\ccy$ and  $\abs{x-x_0}<\ccy$. Hence the definition \eqref{h}
yields $|D^m h(x)|=O(1)$, and the result follows by \eqref{Psi}--\eqref{Phi}
together with the chain rule and Lemmas \ref{LB} and \ref{LD}.
\end{proof}

Note for later than since $G(0)=0$ and $\norm{D G(y)} =O(1)$ in $B_{\ccy}$,
we have
\begin{equation}\label{GB1}
 |G(w,x-x_0)| = O(|(w,x-x_0)|)
\end{equation}
if $(w,x-x_0)\in B_{\ccy}$.

We now analyze the exponential term $e^{N \Psi(n/N)}$ of the 
formula~\eqref{tpnN} for $p_{n,N-n}$, 
which is valid for $\abs{n/N-x_0}<\ccy$.
The next result in particular implies that $\Psi(x) \le 0$ is a concave 
function with a unique maximizer~$\xx$ close to~$x_0$. 
As we shall see, this essentially means that the dominant contribution 
to the sum of the $p_{n,N-n}$ comes from the terms 
with~$n/N$ close to~$\xx$, which is in turn close to~$x_0$.

\begin{lemma}
  \label{LF}
There exist constants
$\ccname\cclfx, \ccname\cclf >0$ with $\cclf\le\cclfx<\frac13\ccy$
such that if\/ $|\E Y-1|\le\cclf$, then the following hold.
\begin{romenumerate}
\item \label{LF0}
If $x\in\bbR$ with $|x-x_0|\le3\cclfx$, then 
\begin{equation}\label{lf0}
  \Psi(x)= -\Omega\bigpar{|\E Y-1|^2+|x-x_0|^2}.
\end{equation}
\item \label{LF'}
There exists $\xx\in\bbR$ with $|\xx-x_0| = O(|\E Y-1|)$ and $|\xx-x_0| < \cclfx$ such that 
$\Psi'(\xx)=0$.
\item   \label{LF''}
$\Psi''(x) = -\Omega(1)$
for every $x$ with $|x-x_0|\le3\cclfx$.
\item \label{LFPhi}
$\Phi(\xx)=\Omega(1)$.
\end{romenumerate}%
As a consequence, $\xx$ is the unique maximum point 
of  $x\mapsto\Psi(x)$ in $[x_0-3\cclfx,x_0+3\cclfx]$.
\end{lemma}

\begin{proof}
  For $|w|,|x-x_0|\le\ccy$, let
  \begin{equation}\label{Psix}
\Psix(w,x):=x\psi(G(w,x-x_0)), 
  \end{equation}
so that $\Psi(x)=\Psix(\E Y-1,x)$.
In the proofs below we assume 
that $\cclfx$ and $\cclf$ are positive constants, chosen 
later, with $\cclf\le\cclfx<\frac13\ccy$,
and that $|w|\le\cclf$ and $|x-x_0|\le 3\cclfx$.

\pfitemref{LF0}  
Since $|w|+|x-x_0| \le 4 \cclfx < 2\ccy \le \ccq$, and $G$ maps $B_{\ccq}$ into $B_{\cckk}$, we have
\begin{equation}\label{Gsmall}
|G(w,x-x_0)|<\cckk.
\end{equation} 
Since $F(0)=0$ and $\norm{D F(y)}=O(1)$ in $B_{\cckk}$,
using $(w,x-x_0)=F(G( w,x-x_0))$ we also have 
$|(w,x-x_0)| = O(|G(w,x-x_0)|)$. This and \refL{Lpsi} 
imply
\begin{equation*}
  \psi(G(w,x-x_0))= -\Omega(|G(w,x-x_0)|^2) = -\Omega(|(w,x-x_0)|^2).
\end{equation*}
Furthermore, as remarked above, 
$|x-x_0|\le3\cclfx < \ccy$ implies $x\ge\ccy$. Hence, recalling \eqref{Psix}, 
\begin{equation}\label{eqPsix}
  \Psix(w,x) 
=  -\Omega(|(w,x-x_0)|^2) = -\Omega\bigpar{w^2+(x-x_0)^2},
\end{equation}
which yields \eqref{lf0} since $\Psi(x)=\Psix(\E Y-1,x)$.

\pfitemref{LF''}  
Using $G(0)=0$, which is shorthand for $G(0,0)=(0,0)$, we have 
\begin{equation}
  \label{blid00}
\Psix(0,x_0)=x_0\psi(G(0,0))=x_0\psi(0,0)=0.
\end{equation}
Together with \eqref{eqPsix}, 
it follows that, for some constant $c>0$, 
\begin{align}
  \label{blid0}
D\Psix(0,x_0)& =0, \\
  \label{blid}
D^2\Psix(0,x_0) & \le -c I .
\end{align}
The same proof as for \refL{LE} shows that 
\begin{equation}
  \label{blidm}
D^m\Psix(w,x)=O(1) 
\end{equation}
for every fixed $m\ge0$.
Using \eqref{blidm} with $m=3$ and \eqref{blid}, 
we see that if $\cclfx$ and hence $\cclf\le\cclfx$ is small enough, then 
\begin{equation*}
D^2\Psix(w,x)\le -\tfrac{c}{2} I
\end{equation*}
when $|w|\le\cclf$
and $|x-x_0|\le3\cclfx$.
In particular, recalling $\Psi(x)=\Psix(\E Y-1,x)$, by taking $w=\E Y-1$ we have
\begin{equation}\label{eqPsixx}
 \Psi''(x) \le -\tfrac c2.
\end{equation}

\pfitemref{LF'}  
Similarly, \eqref{blid0} and \eqref{blidm} with $m=2$ imply that
$D\Psix(w,x)=O(|w|+|x-x_0|)$. In particular,
$\Psi'(x_0) =D_2\Psix(\E Y-1,x_0)=O(|\E Y-1|)$.
Hence we may choose $\cclf$ sufficiently small such that $|\E Y-1|\le\cclf$ implies
$|\Psi'(x_0)|\le c\cclfx/3$. Then the mean value theorem and \eqref{eqPsixx} 
imply
$\Psi'(x_0-\cclfx) > 0$ and $\Psi'(x_0+\cclfx) < 0$, 
so $\Psi'(\xx)=0$ for some $\xx \in (x_0-\cclfx,x_0+\cclfx)$. 
Moreover, by the mean value theorem and~\eqref{eqPsixx} 
we also have $|\xx-x_0| \le \frac{2}{c} |\Psi'(x_0)|$, 
so~\ref{LF'} holds. 

\pfitemref{LFPhi}
Since $|\xx-x_0| \le \cclfx < \ccy$, 
by~\eqref{Gsmall} and the definition~\eqref{h} of~$h$ we have
$|h(\xx)|\le\cckk\le\cca$, so \refL{LB}\ref{LBx} applied to $(Y^0,Z^0)$
gives $\tf_0(h(\xx))\ge\frac12\gd$. 
The other factors in~\eqref{Phi} are bounded below, 
using $\xx \le x_0 + \cclfx \le 1+ \cclfx$ and Hadamard's inequality together with \refL{LB}\ref{LBc}, 
and thus~\ref{LFPhi} follows.
\end{proof}

The following technical lemma will be useful for expanding 
the sum of the $p_{n,N-n}$ estimates~\eqref{tpnN} 
around~$n/N \approx \xx$
(it is easy to give a much more precise formula for~$\SX_{2j}$, 
but we do not need this). 

\begin{lemma}
  \label{Lsum}
For $a>0$, $y\in\bbR$ and an integer $j\ge0$, let
\begin{equation}\label{varin}
  \SX_j=\SX_j(a,y):=\sum_{n\in\bbZ} (n-y)^je^{-a(n-y)^2}.
\end{equation}
Then, uniformly for all $0<a\le1$ and $y\in\bbR$,
\begin{equation}\label{rok}
  \SX_0=\sqrt{\frac{\pi}{a}}+ O\bigpar{a\qqw e^{-\pi^2/a}},
\end{equation}
and for every fixed integer $i\ge0$, 
\begin{align}
  \SX_{2i}&=O\bigpar{a^{-i-1/2}}, \label{emma}
\\
\SX_{2i+1}&= O\bigpar{a^{-i-3/2} e^{-\pi^2/a}}.\label{samuel}
\end{align}
\end{lemma}

\begin{proof}
We first consider $\SX_0=\sum_{n \in \bbZ} e^{-a(n-y)^2}$. 
Applying the well-known Poisson summation formula
\cite[(II.13.4)~or~(II.13.14)]{Zygmund} 
and then using the Gaussian integral
$\int_{-\infty}^{\infty}e^{-(ax^2+bx+c)}\dd x = \sqrt{\frac{\pi}a} e^{b^2/(4a)-c}$, 
a short standard calculation yields the identity 
\begin{equation}\label{Poisson}
  \SX_0=\sum_{n\in\bbZ} \int_{-\infty}^{\infty}e^{-a(x-y)^2} e^{-2\pi\ii n x} \dd x
=\sqrt{\frac{\pi}{a}}\sum_{n\in\bbZ} e^{-\pi^2 n^2/a-2\pi\ii n y} , 
\end{equation}
which for $a\le 1$, say, implies \eqref{rok}.
(In fact, \eqref{Poisson} is equivalent to a well-known identity for the
theta function~$\theta_3$, see \cite[(20.7.32)]{NIST}.) 

Moreover, taking the partial derivative of \eqref{varin} with respect to $y$ we
obtain
\begin{equation}\label{vamod}
  \ddy \SX_j(a,y) = -j \SX_{j-1} + 2a \SX_{j+1}.
\end{equation}
In particular, $2a\SX_1=\ddy \SX_0$, and termwise differentiation of the \rhs{} in
\eqref{Poisson} (noting that the main term, $n=0$, is constant) yields 
\begin{equation*}
  \SX_1= O\bigpar{a^{-3/2} e^{-\pi^2/a}}.
\end{equation*}
Repeated differentiation of \eqref{vamod} and induction now yield \eqref{emma}
and \eqref{samuel}.
\end{proof}

We also have to estimate the sum of the~$p_{n,N-n}$ in~\eqref{qN} 
where~$n/N$ is far from~$x_0$. 
Based on simple Chernoff-type arguments, the next result 
shows that their contribution is negligible.  

\begin{lemma}\label{Llarge}
If $|n/N-x_0|\ge \cclfx$, then $p_{n,N-n}\le e^{-\Omega(N)}$. 
\end{lemma}
\begin{proof}
For any $u,v>0$, from \eqref{pab-g} we have
\begin{equation}\label{pabdev}
\an p_{\an,\bm} 
\le
u^{-\an}v^{-\bm}\tg_0(u,v)g(u,v)^\an.
\end{equation}
Take $u=1$ and $v=e^t$, with $|t|\le\log R$, and define
\begin{equation*}
\gam(t):=e^{-t\E Z}g(1,e^t)=\E e^{t(Z-\E Z)}.
\end{equation*}
For any $0 \le n \le N$, \eqref{pabdev} yields 
\begin{equation}\label{magnus}
  np_{n,N-n}\le e^{t(n-N)+tn\E Z} \tg_0(1,e^t) \gamma(t)^n.
\end{equation}
Note that $\gam(0)=1$ and $\gam'(0)=0$. 
Since $g(1,e^t)=f(0,t)$ and $\E Z = D_2f(0,0)$, by \refL{LB}\ref{LBa} 
there is a constant $\CCname\CCTT>0$ 
such that $\gam''(t)\le \CCTT$ whenever $|t| \le \cca$, and so
\begin{equation}\label{mgb} 
\gam(t)\le 1 + \CCTT t^2\le e^{\CCTT t^2}.
\end{equation}
By assumption, $|n-Nx_0|\ge \cclfx N$. 
Recalling that $x_0=1/(1+\E Z)$ and $\E Z \ge 0$, it follows that
\begin{equation}\label{mgc}
|t(n-N+n\E Z)| = |t| \cdot |n(1+\E Z)-N| \ge |t| \cdot \cclfx N(1+\E Z) \ge \cclfx |t| N .
\end{equation}
We now choose $t=\pm c$ where $c:=\min\set{\frac12\cclfx/\CCTT,\cca}$,
and the sign is such that $t(n-N+n\E Z)<0$. 
Using \eqref{magnus}--\eqref{mgc} and $n \le N$, we infer
\begin{equation*}
  n p_{n,N-n} \le \tg_0(1,e^t) \cdot e^{-\cclfx|t|N + \CCTT t^2N} \le O(1) \cdot e^{- \cclfx c N/2} , 
\end{equation*}
completing the proof for $n \ge 1$. 

Finally, in the remaining case $n=0$ we have $|\bp^S|=Z^0$, since $|\bp^L|=0$ if and only if $Y^0=0$. 
Hence
\[
p_{0,N}=\P(Y^0=0, Z^0=N)\le g_{Y^0,Z^0}(1,R) \cdot R^{-N} = O(1) \cdot R^{-N} ,
\]
completing the proof (since $R>1$). 
\end{proof}

We are now ready to prove \refT{T1}.

\begin{proof}[Proof of \refT{T1}]
We suppose throughout that $\ccT\le\cclf$ and that $|\E Y-1|\le\ccT$.

We start by considering the quantities $\xi$ and $\theta$ defined in
\eqref{xi} and \eqref{theta}.
By \refL{LF}, $\Psi(x)$ has a local maximum point $\xx \in (x_0-\cclfx,x_0+\cclfx)$.
As in~\eqref{xi} and~\eqref{theta}, let
\[
 \xi := -\Psi(\xx) \hbox{\qquad and\qquad} \theta := \sqrt{2\pi / |\Psi''(\xx)|}\ \Phi(\xx).
\]
By Lemmas \ref{LE} and \ref{LF}\ref{LF''}, $\Psi''(\xx)=-\Theta(1)$.
By \eqref{lf0} we have $\xi=-\Psi(\xx)=\Omega(|\E Y-1|^2)$. 
Recalling that $\Psi(\xx)=\Psix(\E Y -1,\xx)$, see~\eqref{Psix}, 
by combining \eqref{blid00}, \eqref{blid0} and \eqref{blidm} (with $m=2$) 
together with \refL{LF}\ref{LF'}, it follows that 
\begin{equation}\label{erika}
\xi = |\Psix(\E Y -1,\xx)|
= O\bigpar{|\E Y-1|^2+|\xx-x_0|^2}
= O\bigpar{|\E Y-1|^2}.
\end{equation}
Hence $\xi=\Theta\bigpar{|\E Y-1|^2}$, as claimed.
That $\theta=\Theta(1)$
follows from the bound $|\Psi''(\xx)|=\Theta(1)$ above and
Lemmas \ref{LE} and \ref{LF}\ref{LFPhi}, which give $\Phi(\xx)=\Theta(1)$.

Since $\xi$ and $\theta$, which do not depend on $N$, are both $O(1)$, for any fixed $N$,
\eqref{t1} holds trivially simply
by taking the implicit constant large enough. Thus we may assume throughout that $N^{-0.4}\le\cclfx$.

We have $\qN=\sum_{n=0}^N p_{n,N-n}$.
We estimate this sum by Laplace's method, similarly to the argument in the
proof of \refT{Tpab}, but now for a sum instead of a two-dimensional integral. 

We consider first $n$ such that $|n/N-\xx|< N^{-0.4}$, which includes the main terms in the sum. 
Suppose that $|x-\xx|< N^{-0.4}$. Using \refL{LE}, a Taylor expansion then
yields, 
\cf{}  \eqref{cf}, 
\begin{equation*}
 N\Psi(x)+N\xi = N\Psi(x)-N\Psi(\xx)
 = N\tfrac12 \Psi''(\xx)(x-\xx)^2
 +N\tfrac{1}6 \Psi'''(\xx)(x-\xx)^3 +O\bigpar{N|x-\xx|^4},
\end{equation*}
which by exponentiation and a Taylor expansion of $\Phi(x)$ yields, \cf{}
\eqref{ollo},
{\multlinegap=0pt
\begin{multline*}
  e^{N\Psi(x)}\Phi(x)
  =e^{-N\xi+N\tfrac12 \Psi''(\xx)(x-\xx)^2}
\Bigl(\Phi(\xx)+ \Phi'(\xx)(x-\xx)
+N\Phi(\xx)\tfrac{1}6 \Psi'''(\xx)(x-\xx)^3
\\
+ O\bigpar{|x-\xx|^2+ N|x-\xx|^4+N^2|x-\xx|^6}
\Bigr).
\end{multline*}}%
Similar, but simpler, reasoning also shows that 
if $|x-\xx|< N^{-0.4}$, then
\begin{equation*}
  e^{N\Psi(x)}N^{-1}
  = e^{-N\xi+N\tfrac12 \Psi''(\xx)(x-\xx)^2} \cdot O(N^{-1}).
\end{equation*}
Consequently, since $\Psi''(\xx) \le 0$, if we define
\begin{equation}\label{Sj}
  S_j:=\sum_{|n/N-\xx| < N^{-0.4}} \Bigpar{\frac{n}{N}-\xx}^j
  e^{-\frac12|\Psi''(\xx)|(n-N\xx)^2/N}, 
\end{equation}
then \eqref{tpnN} yields
{\multlinegap=0pt
\begin{multline}\label{ull}
\sum_{|n/N-\xx|< N^{-0.4}}
p_{n,N-n}
  =
N\qww
e^{-N\xi}
\Bigl(\Phi(\xx) S_0+ \Phi'(\xx)S_1
+N\Phi(\xx)\tfrac{1}6 \Psi'''(\xx)S_3
\\
+ O\bigpar{S_2+NS_4+N^2S_6}
+ O\bigpar{N\qw S_0} 
\Bigr).
\end{multline}}%
(The odd sums $S_1$ and $S_3$ do not vanish as the corresponding integrals
in the proof of \refT{Tpab} do, but we shall see that they are exponentially
small.) 
Recall (from the start of the proof) that $\Psi''(\xx)=-\Theta(1)$.
It follows that
if we extend the summation in the definition \eqref{Sj} to all $n\in\bbZ$,
and denote the result by $S_j'$,
then $S_j-S'_j$ is $O\bigpar{e^{-\Omega(N^{0.2})}}$
for each fixed $j$.
Let $a=|\Psi''(\xx)|/2N$.
In the notation of \refL{Lsum},
$S_j'=N^{-j}\SX_j(a,N\xx)$. The error terms of the form $O(a^{-O(1)}e^{-\pi^2/a})$
in the conclusion of \refL{Lsum} are~$e^{-\Theta(N)}$ and so negligible.
Thus, from \refL{Lsum} and \eqref{ull},
recalling the definitions \eqref{xi} and \eqref{theta} of $\xi$ and $\theta$, we find
\begin{equation}\label{winston}
\sum_{|n/N-\xx|< N^{-0.4}}
p_{n,N-n}
  =
N\qww e^{-N\xi}
\bigpar{N^{1/2}\theta+O(N^{-1/2})}.
\end{equation}

Next,
consider $n$ such that $N^{-0.4}\le |n/N-\xx|\le2\cclfx$, and recall that
$3\cclfx < \ccy$.
If $N^{-0.4}\le |x-\xx|\le2\cclfx$, then \refL{LF} implies that 
$|x-x_0| \le 3\cclfx$ and
$\Psi(x)\le\Psi(\xx)-\Omega((x-\xx)^2)\le\Psi(\xx)-\Omega( N^{-0.8}) = -\xi-\Omega(N^{-0.8})$.
Hence, by \eqref{tpnN} and \refL{LE},
if $N^{-0.4}\le |n/N-\xx|\le2\cclfx$, then \refL{LF} implies that 
\begin{equation*}
p_{n,N-n}=N^{-2}e^{-N\xi-\Omega(N^{0.2})} \cdot O(1).
\end{equation*}
The sum over such $n$ is easily absorbed into the error term we are aiming for: we have, say,
\begin{equation}\label{winston1}
\sum_{N^{-0.4}\le|n/N-\xx|\le2\cclfx} p_{n,N-n} = O(N^{-5/2}) \cdot e^{-N\xi}.
\end{equation}

Finally, since $|\xx-x_0| \le \cclfx$ by \refL{LF}\ref{LF'} and $0 \le n \le N$, 
using \refL{Llarge} there exists a constant~$c>0$ such that, say, 
\begin{equation}\label{eleonora}
\sum_{|n/N-\xx|> 2\cclfx}
p_{n,N-n}
\le 
\sum_{|n/N-x_0|> \cclfx}
p_{n,N-n}
 \le
O(N) \cdot e^{-2c N} = O(N^{-5/2}) \cdot e^{-c N}.
\end{equation}
Recalling that $|\E Y-1|\le\ccT$,
by~\eqref{erika} we may choose $\ccT \le\cclf$ sufficiently small so that $\xi < c$, and 
then~\eqref{t1} follows from \eqref{winston}, \eqref{winston1} 
and \eqref{eleonora}.
\end{proof}

\newpage

\section{Application to branching process families}\label{Svary}

In this section we apply the main result of \refS{Smain} (\refT{T1}) to a family of branching processes.
The goal is to prove \refT{Tfinal} below, giving estimates for the point probabilities $\qN$ in a form suitable for the application to Achlioptas processes in~\cite{RW}.

\subsection{Properties of general parameterized families}\label{SvaryGPf}

By a \emph{branching process family} $(\bp_{Y_u,Z_u,Y^0_u,Z^0_u})_{u\in I}$
we simply mean a family of branching processes of the type in Definition~\ref{def:bp},
one for each $u$ in some interval $I\subset \bbR$. Given such a family, we write
\[
 g_u(y,z):= g_{Y_u,Z_u}(y,z) = \E(y^{Y_u}z^{Z_u}) \qquad\text{and}\qquad  g^0_u(y,z):=g_{Y^0_u,Z^0_u}(y,z)=\E(y^{Y^0_u}z^{Z^0_u})
\]
for the corresponding probability generating functions. Note that the branching process family
is fully specified by the interval~$I$ and the functions~$g_u$ and~$g^0_u$.

The following auxiliary
result shows that the associated parameters $\xi_u=\xi_{Y_u,Z_u}$ 
and $\theta_u=\theta_{Y_u,Z_u,Y^0_u,Z^0_u}$ defined as in \refT{T1} vary smoothly in~$u$. 
This will later allow us to compare the parameters~$\xi_{Y,Z}$ 
and~$\theta_{Y,Z,Y^0,Z^0}$ resulting from different probability distributions 
$(Y^0,Z^0)\in \kko$ and $(Y,Z)\in\kk$  
(by integrating linear mixtures that interpolate between them);  
here the extra $|\E Y_u-1| = O(1)$ factor in~\eqref{64} is crucial. 

\begin{lemma}\label{lfamily}
Suppose that $R>1$, $\CM<\infty$, $k_1,k_2 \in \bbN$, and $\delta>0$. 
Set $\kko = \kko(R,\CM,\delta)$ and $\kk = \kk(R,\CM,k_1,k_2,\delta)$. 
Let $(\bp_u)_{u\in I}= (\bp_{Y_u,Z_u,Y^0_u,Z^0_u})_{u\in I}$ be a branching process family such that,
for every $u\in I$, we have $(Y^0_u,Z^0_u)\in \kko$, $(Y_u,Z_u)\in\kk$, and $|\E Y_u-1|\le\ccT$, 
where~$\ccT>0$ is the constant appearing in \refT{T1}. 
Suppose that 
$g_u(y,z)$ and $g_u^0(y,z)$ are 
analytic as functions of $(u,y,z)$ in the domain
\[
\cD_{I,R} := I\times \{(y,z) \in \bbC: \, |y|,|z|< R\} \: \subset \: \bbR\times\bbC^2, 
\]
and that for some $\gl$,
\begin{equation}\label{sjw}
\max\Bigl\{ \Bigabs{\pddu g_u(y,z)}, \; \Bigabs{\pddu g_u^0(y,z)} \Bigr\} \le \gl
\end{equation}
for all $(u,y,z) \in \cD_{I,R}$. 
Let
\[
 \xi_u :=\xi_{Y_u,Z_u} \qquad\text{and}\qquad  \theta_u :=\theta_{Y_u,Z_u,Y^0_u,Z^0_u}
\]
be defined as in \refT{T1}. Then $\xi_u$ and $\theta_u$ are (real) analytic as functions of $u\in I$.
Furthermore, 
\begin{align}
\label{64}
  \ddu \xi_u &= O\bigpar{\gl|\E Y_u-1|},\\
\label{63}
 \ddu \theta_u & =O(\gl) ,
\end{align}
where the implicit constants in~\eqref{64} and~\eqref{63} depend only on $R,\CM,k_1,k_2,\delta$. 
\end{lemma}
\begin{proof}
By assumption, the conditions of \refT{T1} hold for each $u\in I$.
For any of the quantities or functions defined in previous sections for a single
branching process, we use a subscript $u$ to denote the corresponding quantity
or function associated to $\bp_u$.
As in previous sections,
$\ga$ and $\gb$ always denote real numbers.

The idea of the proof is as follows. For a given $u$, the functions defined in the previous 
sections are defined, either explicitly or implicitly, in terms of $g_u$ and $g_u^0$ (or
their reparameterizations $f_u$ and $f_u^0$). Roughly speaking, since $g_u$ and $g_u^0$
vary analytically in $u$ by assumption (and with $u$-derivative $O(\gl)$), it follows
that the same is true for the derived quantities. There are various steps where
we must be slightly careful; for example, when taking logs (there is no problem
as we stick to the domain $|z-1|\le \frac12$), or dividing by the square
root of a certain determinant (there is no problem since this determinant is $\Omega(1)$
by \refL{LD}). We must also be careful with the implicit definitions of $G_u$ and $\xx_u$;  
the hardest part of the argument is to establish~\eqref{64} with $O(\gl|\E Y_u-1|)$ instead of~$O(\gl)$. 

Turning to the details, from \eqref{sjw} and standard Cauchy estimates 
we see that for each fixed $m$ we have
\begin{equation}\label{sjw2}
 \Bignorm{D^m \pddu g_u(y,z)} = O(\gl) \qquad\text{and}\qquad \Bignorm{D^m \pddu g_u^0(y,z)} =O(\gl)
\end{equation}
whenever $|y|,|z|\le R^{1/2}$, say.
(Here and below, $D$ does not include derivatives with respect to $u$.)
Since $\cckk\le\cca\le(\log R)/2$, the same estimates hold for the derivatives of~$f_u(y,z)=g_u(e^y,e^z)$ and~$f_u^0(y,z)=g^0_u(e^y,e^z)$ 
in the domain $B_{\cckk}\subset \bbR^2$; from now on we work over the reals.
Recalling the definition~\eqref{F} and $\gf_u = \log f_u$, from~\eqref{sjw2} it follows that
$\norm{\pddu F_u(\ga,\gb)}=O(\gl)$ for~$(\ga,\gb)\in B_{\cckk}$.

From the definition \eqref{F}, the function
$F_u(\ga,\gb)$ is a (real) analytic function of $(u,\ga,\gb) \in I\times B_{\cckk}$.
For each~$u \in I$, by~\eqref{FG} we have an inverse $G_u:B_{\ccq}\to B_{\cckk}$ of the 2-variable function $F_u$.
Applying a standard version of the implicit function theorem locally, we see that $G_u(\ga,\gb)$
is analytic as a function of~$(u,\ga,\gb) \in I \times B_{\ccq}$.\footnote{Fix $u_0$ and $y_0=(\ga_0,\gb_0)$ in the relevant domain. By \eqref{DF}, $DF_u$ is invertible
at $x_0=G_{u_0}(y_0)$, so there is an analytic function~$\hG_u(y)$ defined in a neighbourhood
of $(u_0,y_0)$ such that $F_u(\hG_u(y))=y$. By local uniqueness, $G_u(y)=\hG_u(y)$ near~$(u_0,y_0)$,
so~$G$ is indeed analytic at this point.}

Noting $\E Y_u = \ddy g_u(y,z)\big|_{y=z=1}$, by definition~\eqref{h} and $|\E Y_u-1|\le\ccT$ 
it follows that~$h_u(x)$ is an analytic function of $u,x$ for $u \in I$ and
$|x-x_{0,u}|< \ccy$; we consider in the sequel only such $u$ and $x$.
Inspecting the definitions~\eqref{Psi} and~\eqref{Phi}, using~\refL{LD} 
(to ensure that the determinant is not degenerate), 
we see that~$\Psi_u(x)$ and~$\Phi_u(x)$ are well-defined compositions of analytic functions, 
and thus analytic as functions of~$u,x$.

Since $F_u(G_u(y))=y$ is independent of $u$, writing $x=G_u(y)$ and differentiating yields
$\pddu F_u(x) + DF_u(x)(\pddu G_u(y))=0$ 
and thus, recalling \eqref{DFi},
for $y\in B_{\ccq}$,
\begin{equation}
  \label{pg}
 \Bignorm{\pddu G_u(y)}  \le \Bignorm{DF_u(x)\qw} \cdot \Bignorm{\pddu F_u(x)}=O(\gl). 
\end{equation}
Recalling the definition~\eqref{h}, note that~\eqref{pg} implies $\norm{\pddu h_u(x)}=O(\gl)$. 
Since $\psi_u$ is defined in terms of $\gf_u = \log f_u = \log f_{Y_u,Z_u}$ and its derivatives, see~\eqref{psi}, 
using $(Y_u,Z_u) \in \cK$ it follows that $\norm{D\psi_u(\alpha,\beta)}=O(1)$. 
Furthermore, since estimates analogous to~\eqref{sjw2} 
also hold for~$f_u=f_{Y_u,Z_u}$, we have $\pddu\psi_u(y)=O(\gl)$.
Hence, recalling \eqref{Psi} and writing $y=h_u(x)$,
we have 
\begin{equation}
  \pddu \Psi_u(x) = x\left(\pddu\psi_u(y) 
+ D\psi_u(y)\Bigpar{\pddu h_u(x)}\right) 
= O(\gl) .
\end{equation}
Recalling the definitions \eqref{Psi} and~\eqref{Phi}, and the estimates in Section~\ref{Ssum}, 
we similarly deduce $\pddu \Phi_u(x) =O(\gl)$, $\pddu \Psi'_u(x) =O(\gl)$ and $\pddu \Psi_u''(x) =O(\gl)$. 

Since $\xx_u$ is defined by $\Psi_u'(\xx_u)=0$, we have 
$(\pddu\Psi_u')(\xx_u) + \Psi''_u(\xx_u)\ddu\xx_u=0$. It follows, using the
lower bound $\Psi_u''(\xx_u) = - \Omega(1)$ from \refL{LF}\ref{LF''} and the implicit function
theorem, that $\xx_u$ is an analytic function of $u$, and that 
\begin{equation}\label{dduxxu}
\ddu\xx_u = \frac{-(\pddu\Psi_u')(\xx_u)}{\Psi''_u(\xx_u)}  = O(\gl).
\end{equation} 
As~\eqref{pg} implies $\norm{\pddu h_u(x)}=O(\gl)$, and $\norm{h'_u(x)}=O(1)$ by~\eqref{DmG},
it follows that 
\begin{equation}\label{dh}
\Bignorm{\ddu h_u(\xx_u)} \le  \Bignorm{\Bigl(\pddu h_u\Bigr)(\xx_u)} + \Bignorm{h'_u(\xx_u)} \cdot \Bigabs{\ddu\xx_u} =O(\gl).  
\end{equation}

Similarly, from the definitions \eqref{xi} and \eqref{theta}, using
\eqref{dduxxu} and the bounds above on $\pddu\Psi_u$, $\pddu\Psi_u''$ and~$\pddu\Phi_u$ 
it follows  that $\xi_u$ and $\theta_u$ are
analytic functions of $u$, with $\ddu \xi_u=O(\gl)$ and $\ddu\theta_u = O(\gl)$. It remains only to
establish~\eqref{64}.

For this final step, 
recalling 
\eqref{psi} and
$\gf_u(\ga,\gb) = \log f_u(\ga,\gb)=\log g_u(e^\ga,e^\gb)$, 
note that $\norm{D \pddu \psi_u(x)}=O(\gl)$ follows from~\eqref{sjw2}. 
Since
$\pddu\psi_u(0)=0$, we thus 
have $\pddu \psi_u(x)=O\bigpar{\gl|x|}$.
Similarly, as a consequence of \refL{Lpsi},
$\psi_u(x)=O\bigpar{|x|}$
and 
$\norm{D\psi_u(x)}=O\bigpar{|x|}$.
Writing $\yy_u$ for $h_u(\xx_u)$, it follows from the definition \eqref{Psi} and \eqref{dduxxu}--\eqref{dh} that 
\begin{equation*}
\begin{split}
 - \ddu \xi_u & = \ddu \Psi_u(\xx_u) = \ddu\left( \xx_u\psi_u(\yy_u) \right) \\
& = \left(\ddu \xx_u\right)\psi_u(\yy_u) + \xx_u\left(\pddu\psi_u(\yy_u) 
  + D\psi_u(\yy_u)\Bigpar{\ddu h_u(\xx_u)}\right)
 = O\bigpar{\gl|\yy_u|}.
\end{split}
\end{equation*}
Since $\yy_u=h_u(\xx_u)$, from the definition \eqref{h} of $h_u$, the bound \eqref{GB1}, and, for the final step, \refL{LF}\ref{LF'}, 
it follows that
\[
\ddu \xi_u
= O\Bigpar{\gl\bigpar{|\E Y_u-1|+|\xx_u-x_{0,u}|}} 
= O\bigpar{\gl|\E Y_u-1|},
\]
completing the proof of the lemma.
\end{proof}

\subsection{A specific result suitable for application to Achlioptas processes}\label{SStc-critical}

In this section we use \refT{T1} and \refL{lfamily} to prove the case $p_{\cR}=1$, $K=0$ of Theorem~A.10 of~\cite{RW}, 
used there for the analysis of Achlioptas processes. 
To formulate this main application, i.e., our point probability result for certain (perturbed) branching process families, we need some some further definitions.

\begin{definition}\label{def:bpprops}
Let $t_0<\tc<t_1$ be real numbers. The branching process family
$(\bp_t)_{t\in (t_0,t_1)}=(\bp_{Y_t,Z_t,Y^0_t,Z^0_t})_{t\in (t_0,t_1)}$
is \emph{$\tc$-critical} if the following hold:
\begin{romenumerate}
  \item\label{def:bp:analytic} There exist $\delta>0$ and $R > 1$ with $(\tc-\delta,\tc+\delta) \subseteq (t_0,t_1)$ such that the
probability generating functions
	\begin{equation}\label{fdef}
 g_t(y,z) :=  g_{Y_t,Z_t}(y,z) = \E\bigl(y^{Y_t}z^{Z_t}\bigr) \qquad\text{and}\qquad
 g^0_t(y,z) := g_{Y^0_t,Z^0_t}(y,z) = \E\bigl(y^{Y_t^0}z^{Z_t^0}\bigr) 
\end{equation}
are defined and analytic on the domain 
\[
 \{ (t,y,z)\in \bbR\times \bbC^2 : \: |t-\tc| < \delta \hbox{ and } |y|,|z| < R\}.
\]
\item\label{def:bp:crit} We have
\begin{equation}
\label{means:Y}
 \E Y_{\tc} =1, \qquad \E Y^0_{\tc} >0 
\qquad\text{and}\qquad 
  \left. \ddt \E Y_{t} \right|_{t=\tc}  > 0 . 
\end{equation}
\item\label{def:bp:nondeg} There exists some $k_0\in\bbN$ such that
\begin{equation}\label{cnondeg}
\min\Bigl\{
\P\bb{Y_{\tc}=k_0,\,Z_{\tc}=k_0}, \: 
\P\bb{Y_{\tc}=k_0+1,\,Z_{\tc}=k_0}, \:
\P\bb{Y_{\tc}=k_0,\,Z_{\tc}=k_0+1}
\Bigr\} > 0 .
\end{equation}
\end{romenumerate}
\end{definition}

\begin{definition}\label{def:dtype}
Let $(\bp_t)_{t\in (t_0,t_1)}$ be a $\tc$-critical branching process family,
and let $\delta$, $R$ and $k_0$ be as in Definition~\ref{def:bpprops}.
Given $t,\eta \ge 0$ with $|t-\tc|< \delta$, 
we say that the branching process $\bp_{Y,Z,Y^0,Z^0}$ is \emph{of type $(t,\eta)$}
(with respect to $(\bp_t)$, $\delta$, $R$, and $k_0$) 
if the following hold:
\begin{romenumerate}
  \item\label{def:bp2:analytic} Writing $\cN := \{(y,z) \in \bbC^2: \, |y|,|z|<R\}$, 
the expectations
	\begin{equation}\label{tfdef}
 \tg(y,z):=g_{Y,Z}(y,z) = \E\bigl(y^Yz^Z\bigr) \qquad\text{and}\qquad
 \tg^0(y,z) := g_{Y^0,Z^0}(y,z) = \E\bigl(y^{Y^0}z^{Z^0}\bigr) 
\end{equation}
are defined (i.e., the expectations converge absolutely) for all $(y,z) \in \cN$.
\item\label{def:bp2:close} For all $(y,z) \in \cN$ we have
\begin{equation}\label{dtype}
  \bigl|\tg(y,z)-g_t(y,z)\bigr|\le\eta  \qquad\text{and}\qquad
 |\tg^0(y,z)-g^0_t(y,z)\bigr|  \le \eta .
\end{equation}
\end{romenumerate}
\end{definition}

Note that $\bp_t=\bp_{Y_t,Z_t,Y^0_t,Z^0_t}$ is itself of type $(t,\eta)$ for any $\eta\ge 0$. 
The following result relates the point probabilities
from~$\bp_t$ with those from branching processes $\bp$ of type $(t,\eta)$. 
A key feature is the form of the uniform $O(\eta|t-\tc|+\eta^2)$
error term in~\eqref{txbds}.
In~\eqref{eq:asyfull} and~\eqref{txbds} below, 
we have $\xi_{Y_t,Z_t} = \psi(t) = \Theta((t-\tc)^2)$ and $\theta_{Y_t,Z_t,Y^0_t,Z^0_t} = \theta(t) = \Theta(1)$ 
for $\bp=\bp_t$ (using $\eta=0$), 
and $\xi_{Y,Z} \sim \psi(t)$ and $\theta_{Y,Z,Y^0,Z^0} \sim \theta(t)$ for any branching process~$\bp$ of type~$(t,\eta)$ with $\eta \ll |t-\tc| \le \eps_0$. 
In the near-critical case~$t = \tc \pm \eps$, the size--$N$ point probabilities of~$\bp_t$ and~$\bp$ thus both decay exponentially in~$\Theta(\eps^2N)$.

\begin{theorem}[Point probabilities of $\bp$ of type $(t,\eta)$]\label{Tfinal} 
Let $(\bp_t)_{t\in (t_0,t_1)}$ be a $\tc$-critical branching process family. 
Then there exist constants~$\eps_0,\eta_0>0$ and analytic functions~$\theta$,~$\psi$
on the interval $I=[\tc-\eps_0,\tc+\eps_0]$ 
such that 
\begin{equation}\label{eq:asyfull}
 \P(|\bp|=N) = (1+O(1/N)) N^{-3/2} \theta_{Y,Z,Y^0,Z^0}\  e^{-\xi_{Y,Z} N}
\end{equation}
uniformly over all $N\ge 1$, $t\in I$, $0\le\eta\le \eta_0$ and
all branching processes $\bp=\bp_{Y,Z,Y^0,Z^0}$ of type $(t,\eta)$ 
(with respect to $(\bp_t)$),
where the parameters $\xi_{Y,Z}$ and $\theta_{Y,Z,Y^0,Z^0}$, which depend on the distributions of~$(Y,Z)$
and of~$(Y^0,Z^0)$, satisfy
\begin{equation}\label{txbds}
 \xi_{Y,Z} =  \psi(t) +O(\eta|t-\tc|+\eta^2) \qquad\text{and}\qquad \theta_{Y,Z,Y^0,Z^0}=\theta(t)+O(\eta).
\end{equation}
Moreover, $\theta(t)>0$, $\psi(t)\ge 0$, $\psi(\tc)=\psi'(\tc)=0$, and $\psi''(\tc)>0$.
\end{theorem}

\begin{proof}
Fix a $\tc$-critical branching process family $(\bp_t)_{t\in (t_0,t_1)}$,
and let $\delta>0$ and $R>1$ be as in the definitions above.
We pick $0<\eps_0<\delta$, and decrease $R$ slightly, keeping $R>1$.
Then $g(t,y,z)=g_t(y,z)$ and $g^0(t,y,z)=g^0_t(y,z)$ 
are continuous on the compact domain $|t-\tc|\le \eps_0$,
$|y|, |z|\le R$ and so bounded, say by~$M_1$. Let~$M=M_1+1$. Then, provided $|t-\tc|\le \eps_0$, 
by~\eqref{dtype} any $\bp$ of type~$(t,\eta)$ with $\eta\le 1$ satisfies 
\begin{equation*}
 \max\Bigl\{ \bigl|\tg(y,z)\bigr|, \; \bigl|\tg^0(y,z)\bigr|\Bigr\}\le M \qquad\text{whenever}\quad |y|,|z|\le R.
\end{equation*}

For any integers $k,\ell$, we have 
\[
 \P(Y_t=k,\,Z_t=\ell)= \left. \frac{1}{k!\ell!} \cdot \frac{\partial^{k+\ell}}{\partial y^k \partial z^\ell}g_t(y,z)\right|_{y=z=0}.
\]
Since $g_t(y,z)$ is analytic in $t,y,z$, this probability varies continuously in $t$.
Moreover, since $\P(Y=k,\,Z=\ell)$ can analogously be written as a derivative of $\tg$ evaluated at $(0,0)$, 
using standard Cauchy estimates and~\eqref{dtype} we infer 
\begin{equation*}
   \bigl| \P(Y=k,\,Z=\ell) - \P(Y_t=k,\,Z_t=\ell) \bigr| = O(\eta).
\end{equation*}
A similar argument shows that $\E Y_t= \ddy g_t(y,z)\big|_{y=z=1}$ is continuous in $t$, 
and that Cauchy's estimates imply
\begin{equation}\label{be}
 \bigl| \E Y_t-\E Y \bigr| = O(\eta).
\end{equation}
Analogous reasoning applies to $\E Y^0_t$ and $\E Y^0$.

By definition of a $\tc$-critical branching process family, there is some $\delta>0$
such that for $t=\tc$ all of $\E Y_t^0$, $\P(Y_t=k_0,Z_t=k_0)$, $\P(Y_t=k_0+1,Z_t=k_0)$
and $\P(Y_t=k_0,Z_t=k_0+1)$ are at least $2\delta$, say. Furthermore,
at $t=\tc$ we have $\E Y_t=1$. From the argument above these quantities all vary
continuously in $t$, and change by $O(\eta)$ when we move from $\bp_t$ to some $\bp$
of type $(t,\eta)$. It follows that there is a constant $\eta_0>0$ such that,
after reducing $\eps_0$ if necessary, whenever $|t-\tc|\le \eps_0$ and $\eta\le \eta_0$,
then any $\bp$ of type $(t,\eta)$ satisfies the conditions of \refT{T1},
namely that $(Y^0,Z^0)\in \kko$, $(Y,Z)\in \kk$ and $|\E Y-1|\le \ccT$.

Now, applying \refT{T1} to each branching process in the family $(\bp_t)_{t\in [\tc-\eps_0,\tc+\eps_0]}$,
and \refL{lfamily} to the family itself,
establishes the $\eta=0$ case of \refT{Tfinal} with $\theta(t)=\theta_t$ and $\psi(t)=\xi_t$.
Indeed, \refT{T1} gives that $\theta=\Theta(1)$, so we do have $\theta(t)>0$,
while \eqref{xix} gives $\psi(t)=\xi_t=\Theta(|\E Y_t-1|^2)$,
which is $\Theta(|t-t_c|^2)$ since~\eqref{means:Y} implies,
after reducing $\eps_0$ if necessary, that
\begin{equation}\label{be1}
 \bigl| \E Y_t-1 \bigr| = \Theta(|t-t_c|).
\end{equation}
It follows that $\psi(\tc)=\psi'(\tc)=0$ and $\psi''(\tc)>0$.

To complete the proof, assume now that
$\bp=\bp_{Y,Z,Y^0,Z^0}$ is of type $(t,\eta)$, with $0\le \eta\le\eta_0$ and $|t-\tc|\le\eps_0$.
As noted above, \refT{T1} applies to $\bp$, giving~\eqref{eq:asyfull}; it remains to establish~\eqref{txbds}.  
We do this by interpolating between~$\bp$ and~$\bp_t$, and applying Lemma~\ref{lfamily}. 
Consider the branching process family
$(\bY_u,\bZ_u,\bY^0_u,\bZ^0_u)_{u\in [0,1]}$ defined by the mixtures 
\begin{equation}\label{mixed}
\bg_u(y,z):=(1-u)g_t(y,z)+u \tg(y,z) \qquad \text{and} \qquad \bg^0_u(y,z):=(1-u)g^0_t(y,z)+u \tg^0(y,z).
\end{equation}
(As noted earlier, the probability generating functions $\bg_u$, $\bg^0_u$ and the interval $I=[0,1]$ fully 
specify the family.) 
Since the assumptions of \refT{T1} are preserved by taking mixtures, every branching
process in this family satisfies these assumptions. (In fact, they are all clearly of type $(t,\eta)$ too.)
Moreover,
the assumption~\eqref{dtype} implies that \eqref{sjw}
holds with $\gl=\eta$, and 
since $\E \bY_u = \ddy \bg_u(y,z)\big|_{y=z=1}$ we have
\begin{equation}
\label{mixe}  
|\E \bY_u-1| \le |\E Y_t -1| + u|\E Y-\E Y_t|=O(|t-t_c|+\eta)
\end{equation}
by~\eqref{be} and~\eqref{be1}.
Thus we may apply Lemma~\ref{lfamily}, and, by integrating~\eqref{64} with $\xi_u=\xi_{\bY_u,\bZ_u}$, we infer 
\begin{equation}\label{mixint}  
 \xi-\xi_t = \xi_{\bY_1,\bZ_1} - \xi_{\bY_0,\bZ_0} = O\bigpar{\eta|t-t_c|+\eta^2}.
\end{equation}
Finally, $\theta-\theta_t=O(\eta)$ follows similarly by integrating~\eqref{63}.
\end{proof}

\refT{Tfinal} immediately implies the key case $p_{\cR}=1$, $K=0$ of Theorem A.10 of~\cite{RW} 
with any positive value of the constant~$c$. 
Indeed, after reducing $\eps_0$ if necessary, 
the assumption $\eta\le c|t-\tc|$ in the latter theorem implies the assumption $\eta\le\eta_0$ of \refT{Tfinal}.
Moreover, the same assumption $\eta\le c|t-\tc|$ together with~\eqref{txbds} implies the bound $\xi_{Y,Z}=\psi(t)+O(\eta|t-\tc|)$
in Theorem A.10 of~\cite{RW}.

\section{The survival probability}\label{Ssurvival}

In this section we study the survival probability of the branching process $\bp=\bp_\YYZZ$ from \refD{def:bp}
and the branching process family $(\bp_{u})_{u\in I}$ from \refS{Svary}.  
The goal is to prove \refT{thsurv} below, i.e., to give estimates for $\P(|\bp|=\infty)$ 
suitable for the application to Achlioptas processes in~\cite{RW}.

Our strategy mimics the general approach used in Sections~\ref{Smain}--\ref{Svary} for point probabilities, though the technical details are much simpler. 
In \refS{Ssurvsf} we first prove a technical result for the survival probability $\P(|\bp|=\infty)$ of a single branching process (\refL{LSA}). 
Then we show that in a branching process family $(\bp_{u})_{u\in I}$ certain parameters related to the survival probability vary smoothly in~$u$ (\refL{LSB}). 
Finally, in \refS{Ssurvap} we combine these two auxiliary results to prove \refT{thsurv}.

\subsection{Properties of a single process and general parameterized families}\label{Ssurvsf}

As far as the survival of $\bp_\YYZZ$ is concerned, particles of type~$S$ are
irrelevant and may be ignored, so we may consider a standard single-type
Galton--Watson branching process with offspring distribution~$Y$ and initial
distribution~$Y^0$, which we henceforth denote by $\bp_{\YY}$. 
Thus 
\begin{equation}\label{rhoYY:def}
\P(|\bp_\YYZZ|=\infty)=\P(|\bp_\YY|=\infty) =: \rho_\YY .
\end{equation}
Writing $\mathbbm{1}$ as shorthand for the distribution with constant value one, it similarly follows that 
\begin{equation}\label{rhoY:def}
\P(|\bp^1_{\YZ}|=\infty)=\P(|\bp_{Y,\mathbbm{1}}|=\infty) =: \rho_Y .
\end{equation}
Throughout this section, 
we shall work with the univariate \pgf{s} $g_Y(y):=\E y^Y=g_\YZ(y,1)$ and
$g_{Y^0}(y):=g_{Y^0,Z^0}(y,1)$.
By standard branching process arguments (see, e.g., \cite[Theorem~5.4.5]{GrimStir}), we have
\begin{equation}\label{rhoYY}
  1-\rho_{\YY}=g_{Y^0}(1-\rhoY),
\end{equation}
where the extinction probability $1-\rhoY$ is the smallest non-negative
solution to 
\begin{equation} \label{rhoY}
  1-\rhoY = g_Y(1-\rhoY) .
\end{equation}
Fix $R>1$, $\CM$, $k_1$, $k_2$ and $\delta>0$.
We henceforth assume that $(Y,Z)\in\kk = \kk(R,\CM,k_1,k_2,\delta)$ and
$(Y^0,Z^0) \in \kko = \kko(R,\CM,\delta)$. 
Since $(Y,Z)\in\kk$, by \eqref{RYZ} the function $g_Y(y)$ is analytic in $\set{y\in\bbC:|y|<R}$, with~$g_Y(1)=1$. 
A Taylor expansion of $g_Y(y)$ at $y=1$ yields, for $|x|<R-1$,
\begin{equation}\label{gY}
  g_Y(1-x)=\E(1-x)^Y = \sumno(-1)^n\E \binom{Y}{n} x^n = 1-\E Y x + \E \binom Y2 x^2 - \E \binom Y3 x^3+\cdots . 
\end{equation}
Define 
\begin{equation}\label{hYdef}
  h_Y(x) := \frac{1-g_Y(1-x)}{x}, 
\end{equation}
removing the removable singularity at $x=0$. Then $h_Y$ is analytic in $\set{x\in\bbC:|x|<R-1}$, and
\begin{equation}\label{hY}
 h_Y(x) = \sumno(-1)^n\E \binom{Y}{n+1} x^n
= \E Y - \E \binom Y2 x + \E \binom Y3 x^2+\cdots.
\end{equation}
Observe that if $\rhoY>0$, then~\eqref{rhoY} is equivalent to $h_Y(\rhoY)=1$. 
Furthermore,
\begin{equation}\label{hh'}
  h_Y(0)=\E Y=g'_Y(1),
\qquad
-h_Y'(0)=\E\binom Y2 = \frac{ \E (Y(Y-1))}2.
\end{equation}
We next derive bounds on the derivatives of~$h_Y$ valid for small~$x$.

\begin{lemma}\label{LS}
Suppose that $R>1$, $\CM<\infty$, $k_1,k_2 \in \bbN$, and $\delta>0$. 
There exist constants $0<\ccname\ccLS\le\min\{R-1,1\}/3$
and $\CCname\CCLS\xm$ such that 
if $(Y,Z)\in\kk=\kk(R,\CM,k_1,k_2,\delta)$, then the following hold. 
\begin{romenumerate}
\item \label{LSD}
If $m \in \bbN$ and $|x|\le\ccLS$, then $|D^m h_Y(x)|\le \CCLS\xm$.
\item \label{LS'}
If $\E Y\ge 1-\gd$, then
$h_Y'(0)\le-\gd$ and $\P(Y \ge 2)>0$. 
\item \label{LS'x}
If $\E Y\ge 1-\gd$ and $|x|\le\ccLS$, then
$h_Y'(x)\le-\gd/2$.
\end{romenumerate}
\end{lemma}
\begin{proof}
\pfitemref{LSD}
By~\eqref{hYdef}
and~\eqref{RYZ}, $h(x)=O(1)$ if $|x|=(R-1)/2$, say. 
Hence the result, with $\ccLS:=\min\{R-1,1\}/3$, say, follows by Cauchy's estimates.

\pfitemref{LS'}
If \eqref{a2} holds with $k_1\ge1$, then $\P(Y\ge2)\ge\P(Y=k_1+1) \ge \pi_{k_1+1,k_2}\ge\gd$,
and thus $\E(Y(Y-1))\ge 2\gd$.

If instead \eqref{a2} holds with $k_1=0$, then $\P(Y=0)\ge \pi_{k_1,k_2} + \pi_{k_1,k_2+1}\ge2\gd$.
Since $\E Y\ge 1-\gd$, then $Y \in \bbN$ implies $\E (\indic{Y \ge 2}(Y-1)) = \E(Y-1) +\P(Y=0)\ge\gd$,
and thus $\E(Y(Y-1))\ge 2\E(\indic{Y \ge 2}(Y-1))\ge2\gd$. 

In both cases, $h_Y'(0)\le-\gd$ follows by \eqref{hh'}, and $\P(Y \ge 2)>0$ holds, too.  

\pfitemref{LS'x} 
Follows by \ref{LS'} and \ref{LSD} (with $m=2$), replacing
$\ccLS$ by $\min\set{\ccLS,\gd/2\CCLS^{(2)}}$. 
\end{proof}

We next characterize the survival probability~$\rhoY$ in terms of the (unique) solution to $h_Y(\hrho)=1$. 

\begin{lemma}\label{LSA}
Suppose that $R>1$, $\CM<\infty$, $k_1,k_2 \in \bbN$, and $\delta>0$. 
There exists a constant $0 < \ccname\ccLSA \le \delta$ 
such that the following holds.  
If $(Y,Z)\in\kk(R,\CM,k_1,k_2,\delta)$
and $|\E Y-1|<\ccLSA$, then
there is a unique $\hrho=\hrho_Y \in \{x \in \bbR: |x|<\ccLS\}$ such that
\begin{equation*}
h_Y(\hrho)=1.
\end{equation*}
Furthermore, 
$\rhoY=\max\{\hrho,0\}$, 
$\sign(\hrho)=\sign(\E Y-1)$, 
and $|\hrho|=\Theta(|\E Y-1|)$,
where the implicit constants depend only on $R,\CM,k_1,k_2$ and $\delta$. 
\end{lemma}

\begin{proof}
We apply the inverse function theorem, \refL{LI}, with $d=1$, $r=\ccLS$ and 
\[
F(x):=h_Y(x)-\E Y,
\]
using \eqref{hh'} and \refL{LS} to verify the assumptions; we shall ensure that $\ccLSA\le \delta$,
so $|\E Y-1|<\ccLSA$ implies $\E Y \ge 1-\delta$. 
Writing $B_r=B^1_r=\{x\in\bbR:|x|<r\}$ to avoid clutter (as before),
\refL{LI} shows the existence of a constant $\ccLSA>0$, which we may assume
to be at most $\gd$, and an inverse function
$G:B_{\ccLSA}\to B_{\ccLS}$ 
with $F(G(x))=x$ and $G(0)=0$. We define
\[
 \hrho:=G(1-\E Y) ,
\] 
so that $h_Y(\hrho)=F(\hrho)+\E Y=1$. 
Since $\norm{DG(y)}=O(1)$ in~$B_{\ccLSA}$ by \refL{LI}
and $\norm{DF(x)}=O(1)$ in~$B_{\ccLS}$ 
by \refL{LS}\ref{LSD}, using $G(0)=F(0)=0$ we have
$|\hrho|=|G(1-\E Y)|=O(|\E Y-1|)$ and $|\E Y-1|=|F(\hrho)| = O(\hrho)$, 
establishing $|\hrho|=\Theta(|\E Y-1|)$. 

We relate $\hrho$ and $\rhoY$ by a variant of the usual fixed point 
analysis of $g_Y(x)=x$ in~$[0,R]$. 
Since $\P(Y\ge2)>0$ by \refL{LS}\ref{LS'}, $g_Y$ 
is strictly convex on~$[0,R]$, which implies that $g_Y(x)=x$ has at most
two solutions in this interval, and exactly one solution if $\E Y=1$, 
since $g_Y(1)=1$ and $g'_Y(1)=\E Y$.
Now $x=1$ and $x=1-\rhoY \in [0,1]$ are solutions. Since $h_Y(\hrho)=1$, $x=1-\hrho$ is also a solution
(see~\eqref{hYdef}); since $|\hrho| < \ccLS<\min\{R-1,1\}$, we have $1-\hrho\in (0,R)$.

If $\hrho>0$, then $1-\hrho\in(0,1)$ and $1$ are two distinct solutions;
thus $1-\rhoY=1-\hrho$, and $g'_Y(1) > 1$ by strict convexity.  
Similarly, if $\hrho<0$, then $1-\hrho \in (1,R)$ and thus $1-\rhoY=1$, and $g'_Y(1) < 1$ by strict convexity. 
Finally, if $\hrho=0$, then $\E Y = h_Y(0)=h_Y(\hrho)=1$ by~\eqref{hh'},  
so that $1-\rhoY=1$ (since then $x=1$ is the only solution to $g_Y(x)=x$ in~$[0,R]$). 
Hence $\rhoY=\max\{\hrho,0\}$ in all cases. 
It follows also that $\hrho$ is unique, and that~$\hrho$ has the same sign as $\E Y-1=g'_Y(1)-1$. 
\end{proof}

\begin{remark}
Since $F'(0)=h_Y'(0)=-\E(Y(Y-1))/2$, when $\E Y>1$ it follows easily that
$\rhoY=\frac{2(\E Y-1)}{\E(Y(Y-1))}+ O(|\E Y-1|^2)$. In particular
$\rhoY\sim\frac{2(\E Y-1)}{\E(Y(Y-1))}$ as $\E Y\downto 1$, assuming, as always here,
that $(Y,Z)\in\kk$. 
This holds under much weaker conditions on $Y$, see \cite{Hoppe92} and
\cite{Athreya92} for precise conditions; see also \cite[Section~3]{SJ313}.
\end{remark}

We next consider a branching process family $(\bp_u)_{u\in I}= (\bp_{Y_u,Z_u,Y^0_u,Z^0_u})_{u\in I}$ as in \refS{Svary};
as there we indicate the parameter~$u$ by subscripts. 
Thus, for example,
$\hrho_u=\hrho_{Y_u}$ is defined as in~\refL{LSA}, with
$(Y,Z)$ replaced by $(Y_u,Z_u)$.  
Furthermore, in analogy to~\eqref{rhoYY}, we also define
\begin{equation}\label{hrhoYY}
  1-\hrho\YYu := g_{Y^0_u}(1-\hrho_u).
\end{equation}
Thus, by combining~\eqref{rhoYY} with \refL{LSA}, when $\E Y_u\ge1$ we have $\hrho_u=\rho\Yu$ and $\hrho\YYu=\rho\YYu$. 
Mimicking \refL{lfamily}, the following auxiliary result shows that  
$\hrho_u$ and $\hrho\YYu$ both vary smoothly in~$u$. 

\begin{lemma}\label{LSB}
Suppose that $R>1$, $\CM<\infty$, $k_1,k_2 \in \bbN$, and $\delta>0$. 
Set $\kko = \kko(R,\CM,\delta)$ and $\kk = \kk(R,\CM,k_1,k_2,\delta)$. 
Let $(\bp_u)_{u\in I}= (\bp_{Y_u,Z_u,Y^0_u,Z^0_u})_{u\in I}$ be a branching
process family satisfying the assumptions of \refL{lfamily}, 
with $|\E Y_u-1|\le\ccT$ replaced by $|\E Y_u-1|\le\ccLSA$. 
Let~$\hrho_u$ and~$\hrho\YYu$ be defined as in \refL{LSA} and~\eqref{hrhoYY}. 
Then~$\hrho_u$ and~$\hrho\YYu$ are analytic functions of~$u\in I$. Furthermore, 
\begin{equation}\label{drho}
  \ddu \hrho_u = O(\gl),
\qquad
  \ddu \hrho\YYu = O(\gl),
\end{equation}
where the implicit constants depend only on $R,\CM,k_1,k_2$ and $\delta$. 
\end{lemma}

\begin{proof}
Let $h_u(x)=h_{Y_u}(x):=(1-g_u(1-x))/x$ be the equivalent of~\eqref{hYdef} 
for $\bp_u$, again removing the removable singularity at~$x=0$.  
Then $h_u(x)$ is an analytic function of $(u,x) \in I \times \set{x\in\bbC:|x|<R-1}$. 
Note that~\eqref{sjw} implies $|\pddu h_u(x)|=O(\gl)$ if $|x|=(R-1)/3$, say. 
Since $\ccLS \le (R-1)/3$, by the maximum modulus principle (applied with $u$ fixed)
it follows that 
\begin{equation}\label{pdduhu}
\Bigabs{\pddu h_u(x)} \le C\gl
\end{equation}
for all $u \in I$ and $|x| \le \ccLS$.

By \refL{LSA}, for every $u\in I$ there is a unique $\hrho_u \in \bbR$ with $|\hrho_u|<\ccLS$ such that
\begin{equation}\label{hrhou}
  h_u(\hrho_u)=1.
\end{equation}
Since $|h_u'(\hrho_u)|\ge\gd/2$ by \refL{LS}\ref{LS'x} and $|\E Y_u-1|\le\ccLSA \le \delta$, the
implicit function theorem shows that~$\hrho_u$ is an analytic function of~$u\in I$. 
That $\hrho\YYu$ is analytic then follows from~\eqref{hrhoYY}
and the assumption that~$g_u^0(y,z)$ is analytic.
By differentiating \eqref{hrhou} we obtain 
$\frac{\partial h_u}{\partial u}(\hrho_u)+
h_u'(\hrho_u)\cdot\ddu\hrho_u=0$.  
So, using $|h_u'(\hrho_u)|\ge\gd/2$ and~\eqref{pdduhu},
\begin{equation}\label{eleo}
\ddu\hrho_u=-h_u'(\hrho_u)\qw \cdot  \frac{\partial h_u}{\partial u}(\hrho_u)
=O(1)\cdot O(\gl)
=O(\gl).
\end{equation}

Finally, $g_{Y_u^0}'(1-\hrho_u)=O(1)$ follows from~\eqref{RYZ} and Cauchy's estimates (recall that $|\hrho_u| < \ccLS \le (R-1)/3$).  
By differentiating~\eqref{hrhoYY} and then using~\eqref{sjw} and
\eqref{eleo}, we obtain 
\begin{equation*}
\ddu \hrho\YYu = -\Bigl(\pddu g_{Y_u^0}\Bigr)(1-\hrho_u)+ 
g_{Y_u^0}'(1-\hrho_u) \cdot \ddu\hrho_u = O(\gl) + O(1)\cdot O(\gl) =O(\gl),
\end{equation*}
completing the proof. 
\end{proof}

\subsection{A specific result suitable for application to Achlioptas processes}\label{Ssurvap}

We are now ready to prove our main result, concerning the $t$-dependence of the survival probability of $\bp_t$
when $(\bp_t)_{t\in I}$ is a $\tc$-critical branching process family,
as well as the survival probability of branching processes $\bp=\bp_{Y,Z,Y^0,Z^0}$ of type~$(t,\eta)$; see \refS{SStc-critical} for the relevant definitions. 
Two key features are the convergent power series expansion~\eqref{rexp}, and the uniform $O(\eta)$ error term in~\eqref{radd}.
In particular, we have $\trho\sim \rho(\tc+\eps) = \Theta(\eps)$ for any branching process~$\bp$ of type~$(\tc+\eps,\eta)$ with $\eta \ll \eps \le \eps_0$. 
In the supercritical case~$t = \tc + \eps$, the survival probabilities of~$\bp_t$ and~$\bp$ thus both grow linearly in~$\eps$. 

\begin{theorem}[Survival probabilities]\label{thsurv}
Let $(\bp_t)_{t\in (t_0,t_1)}=(\bp_{Y_t,Z_t,Y^0_t,Z^0_t})_{t\in (t_0,t_1)}$ be a $\tc$-critical branching process family.
Then there exist constants $\eps_0,c>0$ with the following properties.
Firstly, the survival probability $\rho(t):=\P(|\bp_t|=\infty)$ is zero for $\tc-\eps_0\le t\le \tc$, and is positive
for $\tc<t\le \tc+\eps_0$.
Secondly, $\rho(t)$ is analytic on $[\tc,\tc+\eps_0]$. More precisely,
there are constants $a_i$ with $a_1>0$ such that
\begin{equation}\label{rexp}
 \rho(\tc+\eps) = \sum_{i=1}^\infty a_i\eps^i
\end{equation}
for $0\le \eps\le \eps_0$.
Thirdly, for any $t$, $\eta$ with $|t-\tc|\le\eps_0$ and $\eta\le c|t-\tc|$,
and any branching process $\bp=\bp_{Y,Z,Y^0,Z^0}$ of type $(t,\eta)$ (with respect to $(\bp_t)$),
the survival probability $\trho:=\P(|\bp|=\infty)$ is zero if $t\le \tc$, and is positive and satisfies 
\begin{equation}\label{radd}
 \trho = \rho(t)+O(\eta)
\end{equation}
if $t>\tc$, where the implicit constant depends only on the family $(\bp_t)$, not
on $t$ or $\bp$.
Moreover, analogous statements hold for the survival probabilities
$\rho_1(t):=\P(|\bp^1_{Y_t,Z_t}|=\infty)$ and $\trho_1:=\P(|\bp^1_{Y,Z}|=\infty)$. 
\end{theorem}
\begin{proof}
We argue as in the proof of \refT{Tfinal}. In particular, we may assume that 
$(Y,Z)\in\kk$ and $(Y^0,Z^0)\in\kko$ for some $R,M,k_1,k_2,\gd$. 
We shall also assume that $c \le 1$.

We consider only $t$ with $ |t-\tc|\le\eps_0$; we may assume
that $\eps_0$ is small enough that this implies 
$t\in(t_0,t_1)$, and, by \eqref{means:Y}, that $\E Y^0_t>0$, and that
\begin{equation}\label{expYt1}
\sign(\E Y_t-1)=\sign(t-\tc) 
\qquad \text{ and } \qquad 
|\E Y_t-1| = \Theta(|t-\tc|) < \ccLSA .
\end{equation}
By~\eqref{rhoY:def} and \refL{LSA} and it follows that $\rho_1(t)=\rhoYt$ is zero for~$\tc-\eps_0\le t\le \tc$, and positive for~$\tc< t\le \tc+\eps_0$. 
Since~$\P(Y_t^0 \ge 1) >0$, now~\eqref{rhoYY:def} and~\eqref{rhoYY} imply an analogous statement for $\rho(t)=\rhoYYt$. 
Lemmas~\ref{LSA} and~\ref{LSB} also imply that 
\begin{equation}\label{raddYYt}
\rho_1(t)=\hrho\Yt \qquad \text{ and } \qquad \rho(t)=\hrho\YYt
\end{equation}
are both analytic for~$\tc\le t\le \tc+\eps_0$.
Hence~\eqref{rexp} holds if~$\eps_0$ is sufficiently small.

Next, for a branching process of type $(t,\eta)$, 
by~\eqref{be} we have $|\E Y_t - \E Y|=O(\eta)$.
Since~$\eta\le c|t-\tc|$, it follows from~\eqref{expYt1} that if~$c$ is small enough, then
$\sign(\E Y-1)=\sign(t-\tc)$.  
Moreover, since $\eta\le c|t-\tc|\le|t-\tc|\le\eps_0$, using~\eqref{expYt1} we also have $|\E Y-1|  < \ccLSA$ if~$\eps_0$ is small enough. 
Mimicking the above reasoning for $\rho_1(t)$ and $\rho(t)$, 
using~\eqref{rhoYY:def}--\eqref{rhoYY} and \refL{LSA} it follows for $\eta\le c|t-\tc|$ that $\trho_1=\rhoY$ and $\trho=\rhoYY$
satisfy $\trho_1=\trho=0$ if~$\tc-\eps_0\le t\le \tc$, and
$\trho_1,\trho>0$ if~$\tc< t\le \tc+\eps_0$; furthermore, 
\begin{equation}\label{raddYY}
\qquad \trho_1=\hrho_{Y} \qquad \text{ and } \qquad \trho=\hrho_{\YY}
\end{equation}
for~$\tc\le t\le \tc+\eps_0$ and $\eta\le c|t-\tc|$. 

Finally, we consider the interpolating branching process family 
$(\bY_u,\bZ_u,\bY^0_u,\bZ^0_u)_{u\in [0,1]}$ defined by~\eqref{mixed}, 
for which, as noted in Section~\ref{SStc-critical}, 
\eqref{sjw} holds with $\gl=\eta$ and $I=[0,1]$.
 Note that \eqref{mixe} and 
$\eta\le|t-\tc|\le\eps_0$ imply $|\E \bY_u-1|<\ccLSA$ 
provided~$\eps_0$ is small enough. 
Integrating~\eqref{drho} of \refL{LSB} over $u\in[0,1]$ similarly
to~\eqref{mixint} in the proof of \refT{Tfinal}, 
using the identities~\eqref{raddYYt}--\eqref{raddYY} 
we infer $\trho_1-\rho_1(t) =\hrho_{Y}-\hrho\Yt = O(\eta)$ and $\trho-\rho(t)=\hrho_{\YY}-\hrho\YYt = O(\eta)$ for $\tc \le t\le \tc+\eps_0$ and $\eta\le c|t-\tc|$,
completing the proof.
\end{proof}

\refT{thsurv} immediately implies the key case $p_{\cR}=1$, $K=0$ of Theorem~A.11 of~\cite{RW}, 
used there for the analysis of Achlioptas processes.

\bigskip{\bf Acknowledgement.} 
The last two authors are grateful to Christina Goldschmidt for useful pointers to the local limit theorem literature,
which were helpful for the developing parts of the slightly more involved (large deviation based) point probability analysis contained in an earlier version of~\cite{RW}.

\newcommand\AAP{\emph{Adv. Appl. Probab.} }
\newcommand\JAP{\emph{J. Appl. Probab.} }
\newcommand\JAMS{\emph{J. \AMS} }
\newcommand\MAMS{\emph{Memoirs \AMS} }
\newcommand\PAMS{\emph{Proc. \AMS} }
\newcommand\TAMS{\emph{Trans. \AMS} }
\newcommand\AnnMS{\emph{Ann. Math. Statist.} }
\newcommand\AnnPr{\emph{Ann. Probab.} }
\newcommand\CPC{\emph{Combin. Probab. Comput.} }
\newcommand\JMAA{\emph{J. Math. Anal. Appl.} }
\newcommand\RSA{\emph{Random Struct. Alg.} }
\newcommand\ZW{\emph{Z. Wahrsch. Verw. Gebiete} }
\newcommand\DMTCS{\jour{Discr. Math. Theor. Comput. Sci.} }

\newcommand\AMS{Amer. Math. Soc.}
\newcommand\Springer{Springer-Verlag}
\newcommand\Wiley{Wiley}

\newcommand\vol{\textbf}
\newcommand\jour{\emph}
\newcommand\book{\emph}
\newcommand\inbook{\emph}
\def\no#1#2,{\unskip#2, no. #1,} 
\newcommand\toappear{\unskip, to appear}

\newcommand\arxiv[1]{\texttt{arXiv:#1}}
\newcommand\arXiv{\arxiv}

\def\nobibitem#1\par{}


\small
\bibliographystyle{plain}

\normalsize

\end{document}